\documentclass[smallextended]{svjour3}       
\smartqed  
\usepackage{graphicx}
\usepackage{subfigure}
\usepackage{mathrsfs}
\usepackage{amssymb}
\usepackage{latexsym}
\usepackage{amsmath}
\usepackage{galois}
\usepackage{bm}
\usepackage{enumerate}
\usepackage{comment}
\usepackage{color,xcolor}
\usepackage{diagbox}

\newcommand{\norm}[1]{\left\Vert#1\right\Vert}
\newcommand{\abs}[1]{\left\vert#1\right\vert}
\newcommand{\inp}[1]{\left\langle#1\right\rangle}
\newcommand{\argmin}{\operatornamewithlimits{arg\,min}}
\newcommand{\argmax}{\operatornamewithlimits{arg\,max}}

\def\d{\mathrm{d}}
\def\diam{\mathrm{diam}}
\def\proj{\mathrm{Proj}}

\newtheorem{assumption}{Assumption}
\newtheorem{algorithm}{Algorithm}

\begin{document}

\title{Pure Characteristics Demand Models and Distributionally Robust  Mathematical Programs with Stochastic Complementarity Constraints\thanks{
This paper is dedicated to the memory of Olvi L. Mangasarian. His contributions to linear complementarity problems have impacted greatly on our research on distributionally robust mathematical programs with stochastic complementarity constraints.}}

\author{Jie Jiang\thanks{Jie Jiang's work was partly supported by China Postdoctoral Science Foundation (Grant No. 2020M673117) and CAS AMSS-PolyU Joint Laboratory of Applied Mathematics.} \and
Xiaojun Chen\thanks{Xiaojun Chen's work was partly supported by Hong Kong Research Grant Council PolyU15300219.}
}

\institute{Jie Jiang \at College of Mathematics and Statistics, Chongqing University, Chongqing, China\\
              \email{jiangjiecq@163.com}
              \and
          Xiaojun Chen \at Department of Applied Mathematics, The Hong Kong Polytechnic University, Hong Kong\\
\email{maxjchen@polyu.edu.hk} }
\titlerunning{Distributionally robust MP with stochastic complementarity constraints}

\vspace{-0.3in}
\date{}

\maketitle

\vspace{-0.2in}

\begin{abstract}
We formulate pure characteristics demand models under uncertainties of probability distributions as distributionally robust mathematical programs with stochastic complementarity constraints (DRMP-SCC). For any fixed first-stage variable and a random realization, the second-stage problem of DRMP-SCC is a monotone linear complementarity problem (LCP). To deal with uncertainties of probability distributions of the involved random variables in the stochastic LCP, we use the distributionally robust approach. Moreover, we propose an approximation problem with regularization and discretization to solve DRMP-SCC, which is a two-stage nonconvex-nonconcave minimax optimization problem. We prove the convergence of the approximation problem to DRMP-SCC regarding the optimal solution sets, optimal values and stationary points as the regularization parameter goes to zero and the sample size goes to infinity. Finally, preliminary numerical results for investigating distributional robustness of pure characteristics demand models are reported to illustrate the effectiveness and efficiency of our approaches.
\end{abstract}


\section{Introduction}

Pure characteristics demand models are widely used in microeconometrics to estimate parameters in utility functions of agents for given prices and production decisions \cite{BP2007pure,SJ2012constrained}. Such models have advantages in inferring consumers' preference and behavior, but face computational challenges to solve the optimization problem with set-valued stochastic equilibrium constraints as follows:
\begin{equation}
\label{CSW}
\begin{array}{cl}
\min\limits_{x\in X}& \frac{1}{2}\inp{x,Hx} + \inp{c,x} \\
\mathrm{s.t.}& A_t\mathbb{E}_P[S_t(x,\xi)]\ni b_t ~\text{for}~t=1,2,\cdots,T,
\end{array}
\end{equation}
where $X\subseteq \mathbb{R}^n$ is a compact and convex set, $H\in\mathbb{R}^{n\times n}$ is a positive semidefinite matrix, $c\in \mathbb{R}^n$, $\xi:\Omega \rightarrow \Xi\subseteq \mathbb{R}^\nu$ is a random vector defined on a probability space $(\Omega,\mathcal{F},\mathbb{P})$ and supported on $\Xi$, $P=\mathbb{P}\comp \xi^{-1}$, $A_t\in\mathbb{R}^{l\times r}$, $b_t\in \mathbb{R}^{l}$, $S_t:\mathbb{R}^n\times \Xi\rightrightarrows \mathbb{R}^{r}$, $T\geq 1$ is the number of markets, $\mathbb{E}_P[S_t(x,\xi)]$ is a Aumann's (set-valued) expectation for multifunction and for given $(x,\xi)\in X\times \Xi$,
\begin{equation}
\label{UT}
S_t(x,\xi):=\argmax \{\inp{s,u_t(x,\xi)}:\inp{e,s}\leq 1, s\geq 0\}.
\end{equation}
Here $u_t:\mathbb{R}^n\times \Xi\rightarrow \mathbb{R}^{r}$ is the consumers' utility function in market $t$ and $e\in \mathbb{R}^{r}$ is the vector with all elements being 1.

To efficiently solve problem (\ref{CSW}), Pang, Su and Lee \cite{PSL2015constructive} characterized consumers' purchase decision in the constraints of problem (\ref{CSW}) by linear complementarity problems and proposed the following quadratic program with stochastic complementarity constraints (QP-SCC):
\begin{equation}
\label{PSL}
\begin{array}{cl}
\min\limits_{x\in X}& \frac{1}{2}\inp{x,Hx} + \inp{c,x} \\
\mathrm{s.t.}& A_t\mathbb{E}_P[s_t(x,\xi)]= b_t~\text{for}~t=1,2,\cdots,T,\\
& 0\leq z_t(x,\xi) \bot M z_t(x,\xi) + q_t(x,\xi)\geq 0 ~\text{for a.e.}~\xi\in \Xi,
\end{array}
\end{equation}
where
\begin{equation*}
M=
\begin{pmatrix}
0 & e\\
-e^\top & 0
\end{pmatrix}\in \mathbb{R}^{(r+1)\times (r+1)},\quad
q_t(x,\xi)=
\begin{pmatrix}
-u_t(x,\xi)\\
1
\end{pmatrix}\in  \mathbb{R}^{r+1},
\end{equation*}
\begin{equation*}
s_t(x,\xi) \in  \mathbb{R}^{r}, \quad
z_t(x,\xi)=
\begin{pmatrix}
s_t(x,\xi)\\
\gamma_t(x,\xi)
\end{pmatrix}\in  \mathbb{R}^{r+1}
\end{equation*}
and a.e. is the short for almost everywhere. The pioneered QP-SCC formulation opened a way to develop optimization algorithms for solving  pure characteristics demand models.
In \cite{CSW2015regularized}, Chen, Sun and Wets proposed a penalty approach:
\begin{equation}
\label{CSW-1}
\begin{array}{cl}
\min\limits_{x\in X }& \frac{1}{2}\inp{x,Hx} + \inp{c,x} + \varrho \sum_{t=1}^T\norm{A_t\mathbb{E}_P[s_t(x,\xi)] - b_t}^2 \\
\mathrm{s.t.}& 0\leq z_t(x,\xi) \bot M z_t(x,\xi) + q_t(x,\xi)\geq 0, ~  t=1,2,\cdots,T,
\end{array}
\end{equation}
where $\varrho >0$ is a penalty parameter.

In problems (\ref{CSW})-(\ref{CSW-1}), the probability distribution $P$ is supposed to be known exactly.  However, in practice the true probability distribution can hardly be acquired. This observation motivates us to consider a class of distributionally robust stochastic mathematical programs with complementarity constraints (DRMP-SCC) as follows:
\begin{equation}
\label{DREC}\tag{P}
\begin{array}{cl}
\min& \Phi(x,y):=\theta(x) + \max\limits_{P\in\mathcal{P}}h(\mathbb{E}_P[f(x,y(\xi),\xi)]) \\
\mathrm{s.t.}& x \in X,\, 0\leq y(\xi) \bot M(\xi)y(\xi) + q(x,\xi) \geq 0~\text{for a.e.}~\xi\in \Xi,
\end{array}
\end{equation}
where $y: \Xi \to \mathbb{R}^m$ is a measurable mapping, $\theta: \mathbb{R}^n\rightarrow \mathbb{R}$, $f:\mathbb{R}^n\times\mathbb{R}^m\times\Xi\rightarrow \mathbb{R}^l$, $h:\mathbb{R}^l\rightarrow\mathbb{R}$, $\mathcal{P}\subseteq \mathcal{P}(\Xi)$ with $\mathcal{P}(\Xi)$ denoting the collection of all probability distributions supported on $\Xi$, $q:\mathbb{R}^n \times \Xi\rightarrow \mathbb{R}^m$,
$M:\Xi\rightarrow \mathbb{R}^{m\times m}$ and $M(\xi)$ is positive semidefinite for a.e. $\xi \in \Xi$. For fixed $x$ and $\xi$, a feasible vector $y(\xi)$ is a solution of the monotone LCP. The error bounds for
the monotone LCP and  exact penalty theory for mathematical programms with linear complementarity constraints established by Mangasarian et al. in \cite{M1990error,M1992global,MF1967fritz,MP1997exact,MR1994new,MS1986error} have inspired us to solve problem \eqref{DREC}.

Throughout this paper, we assume that $\theta$, $h$ and $f(\cdot,\cdot,\xi)$ are continuously differentiable. Moreover, $\theta$, $h$, $M$,  $q$ and $f(\cdot,\cdot,\xi)$ are Lipschitz continuous with Lipschitz moduli $L_\theta$, $L_h$, $L_M$, $L_q$ and $L_f(\xi)$ satisfying $\max_{P\in {\cal P}} \mathbb{E}_P[L^2_f(\xi)]<\infty$,  respectively. We also assume that problem \eqref{DREC} satisfies the relatively complete recourse condition that is for every $x\in X$ and a.e. $\xi\in \Xi$, the solution set
$\mathrm{SOL}(M(\xi),q(x,\xi))$ of the complementarity constraints in \eqref{DREC}, denoted by
$\mathrm{LCP}(M(\xi),q(x,\xi))$ is nonempty and $\mathbb{E}_P[f(x,y(\xi),\xi)]$ is well-defined for all $ P \in \mathcal{P}$ and $y(\xi) \in \mathrm{SOL}(M(\xi),q(x,\xi))$.
Moreover, we assume that the ambiguity set $\mathcal{P}$ is defined by a general moment
information as follows:
\begin{equation}
\label{AS}
\mathcal{P}= \left\{P\in\mathcal{P}(\Xi): \mathbb{E}_P[\Psi(\xi)]\in\Gamma \right\},
\end{equation}
where $\Psi$ is a continuous random mapping consisting of vectors and/or matrices with measurable
random components, and $\Gamma$ is a closed convex cone in the Cartesian product of some finite dimensional vector and/or matrix spaces. The ambiguity set defined in \eqref{AS} is a very general form and includes many commonly-used moment ambiguity sets, such as the moment ambiguity set in \cite{DY2010distributionally}. For more examples, we refer to \cite[Examples 3-5]{LPX2019discrete}.

Let $\Xi^k=\{\xi^1,\cdots,\xi^k\}$ be a set of $k$ samples of $\xi$ and  define the discrete approximation of $\mathcal{P}$ by
$$
\mathcal{P}_k = \left\{p\in\mathbb{R}_+^k:
 \sum_{i=1}^k p_i = 1,~\sum_{i=1}^k p_i\Psi(\xi^i)\in\Gamma\right\}.$$

We consider the discrete approximation problem of \eqref{DREC} as follows:
\begin{equation}
\label{DDREC}\tag{$\mathrm{P}_k$}
\begin{array}{cl}
\min &\Phi_k(x,\mathbf{y}):= \theta(x) + \max\limits_{p\in\mathcal{P}_k}h( F(x,\mathbf{y})p ) \\
\mathrm{s.t.} &x\in X, \, 0\leq \mathbf{y}  \bot \mathbf{M}\mathbf{y}  + \mathbf{q}(x) \geq 0,
\end{array}
\end{equation}
where $F(x,\mathbf{y})=(f(x,y(\xi^1),\xi^1), f(x,y(\xi^2),\xi^2),\cdots, f(x,y(\xi^k),\xi^k))$,
\begin{align*}
&\mathbf{y}=
\begin{pmatrix}
y(\xi^1)\\
y(\xi^2)\\
\vdots\\
y(\xi^k)
\end{pmatrix},~
\mathbf{M}=
\begin{pmatrix}
M(\xi^1) & 0 & \cdots &0\\
0 & M(\xi^2) & \cdots &0\\
\vdots & \vdots & \ddots & \vdots \\
0 & 0 & \cdots & M(\xi^k)\\
\end{pmatrix}~\text{and}~
\mathbf{q}(x)=
\begin{pmatrix}
q(x,\xi^1)\\
q(x,\xi^2)\\
\vdots\\
q(x,\xi^k)
\end{pmatrix}.
\end{align*}
Moreover, to develop numerical methods and convergence analysis, we consider the following regularized problems of \eqref{DREC} and \eqref{DDREC}, respectively,
\begin{equation}
\label{R-DREC}\tag{$\mathrm{P}_\epsilon$}
\begin{array}{cl}
\min & \Phi_\epsilon(x,y):=\theta(x) + \max\limits_{P\in\mathcal{P}} h(\mathbb{E}_P[f(x,y(\xi),\xi)]) \\
\mathrm{s.t.}& x\in X,\, 0\leq y(\xi) \bot (M(\xi)+\epsilon I)y(\xi) + q(x,\xi) \geq 0~\text{for a.e.}~\xi\in \Xi
\end{array}
\end{equation}
and
\begin{equation}
\label{DREC-1}\tag{$\mathrm{P}_{\epsilon,k}$}
\begin{array}{cl}
\min & \Phi_{\epsilon,k}(x,\mathbf{y}):=\theta(x) + \max\limits_{p\in\mathcal{P}_k} h( F(x,\mathbf{y})p ) \\
\mathrm{s.t.} &x\in X, \,0\leq \mathbf{y}  \bot (\mathbf{M} + \epsilon I)  \mathbf{y}  + \mathbf{q}(x) \geq 0,
\end{array}
\end{equation}
where $\epsilon>0$ is the regularization parameter and $I$ is the identity matrix with proper dimension.

Since $M(\xi)$ is positive semidefinite for fixed $\xi$, the complementarity problem LCP$(M(\xi)+\epsilon I, q(x,\xi))$ has a unique solution \cite{CPS1992linear}, denoted by $\hat{y}_\epsilon(x,\xi)$. Moreover, from the Lipschitz continuity of $q(\cdot,\xi)$, $\hat{y}_\epsilon(\cdot,\xi)$ is Lipschitz continuous \cite{CX2008perturbation}.  Analogously, LCP$(\mathbf{M} + \epsilon I, \mathbf{q}(x))$ has also a unique solution, denoted by $\hat{\mathbf{y}}_\epsilon(x)$, which is also Lipschitz continuous with respect to (w.r.t.) $x$.

Mathematical programming with equilibrium constraints (MPEC) has been extensively studied    \cite{IS2008active,LPR1996mathematical,MP1997exact}. Structural properties, discrete approximation based on sampling and numerical methods of stochastic MPEC with deterministic probability distribution have been investigated  \cite{LCF2009solving,LF2010stochastic,LSS2016approximation,SX2008stochastic,XY2011approximating}.
To the best of our knowledge, there is little discussion on distributionally robust MPEC.
Moreover, due to the complementarity constraints and the composite structure of the objective function, the minimax  problems \eqref{DREC}, \eqref{R-DREC} and \eqref{DREC-1} are generally nonconvex-nonconcave and their saddle points may not exist.
We will focus on their minimax points, minimizers in $x$-space and corresponding optimality conditions.

The main contributions of the paper are summarized as follows.
\begin{itemize}

\item Inspired by the constructive reformulations of pure characteristics demand models in \cite{PSL2015constructive}, we propose a DRMP-SCC model (\ref{DREC}) under uncertainties of probability distributions. We give the definitions of global and local minimax points to capture the optima of the nonconvex-nonconcave minimax problem (\ref{DREC}). Some sufficient conditions of existence of solutions are derived.

\item Under certain conditions, we prove the convergence of problem (\ref{DREC-1}) to problem (\ref{DREC}) regarding optimal solution sets and optimal values as the regularization parameter $\epsilon \downarrow 0$ and the sample size $k\rightarrow\infty$.

\item We define stationary points of problems (\ref{DREC}), (\ref{DDREC}) and (\ref{DREC-1}) in the block coordinatewise sense and establish the convergence of stationary points of problem (\ref{DREC-1}) to those of problem (\ref{DREC}) as the regularization parameter $\epsilon \downarrow 0$ and the sample size $k\rightarrow \infty$.

\end{itemize}

{\bf Notations.}  $\mathbb{B}$ denotes the closed unit ball centered at original point in the corresponding space. $\mathbb{R}_+^n$ denotes the set of nonnegative vectors in $\mathbb{R}^n$. $\norm{\cdot}$ denotes the Euclidean norm of vectors or the induced matrix norm.
$\diam(X):=\sup_{x,z\in X}\norm{x-z}$ denotes the diameter of $X$. $\d(x,Z)=\inf_{z\in Z}\|x-z\|$ and $\d(X, Z)=\sup_{x\in X}\inf_{z\in Z}\|x-z\|$ for $x,z\in \mathbb{R}^n$ and $X, Z\subseteq \mathbb{R}^n$. $\mathrm{int}(X)$ denotes the interior of $X$.

This paper is organized as follows. In Section \ref{Sec2}, we give some preliminaries on the ambiguity set ${\cal P}$ and its approximation ${\cal P}_k$. In Section \ref{Sec3}, we give definitions of global and local minimax points of problems \eqref{DREC}, \eqref{R-DREC} and \eqref{DREC-1}, and some existence results. After that, we prove the convergence of the solution set and optimal value of problem (\ref{DREC-1}) to those of problem (\ref{DREC}) as $\epsilon \downarrow 0$ and $k\rightarrow \infty$. In Section \ref{Sec4}, we first give definitions of stationary points of problems \eqref{DREC}, \eqref{DDREC} and \eqref{DREC-1} and then study convergence assertions on the stationary points of problem (\ref{DREC-1}) to those of problem (\ref{DREC}). In Section \ref{Sec5}, we report numerical results on the pure characteristics demand model under uncertainties of probability distributions. In Section \ref{Sec6}, we give concluding remarks.

\section{Preliminaries}\label{Sec2}

Note that for each $p\in\mathcal{P}_k$, it uniquely determines a discrete probability distribution
$\sum_{i=1}^k p_i \mathbf{1}_{\xi^i}$,
where $\mathbf{1}_{\xi^i}(\cdot)$ is the indicator function, namely, $\mathbf{1}_{\xi^i}(\xi)=1$ if $\xi=\xi^i$; $0$ otherwise. Thus, in what follows, we will not distinguish $p$ and the corresponding probability distribution. In other words, we can write
$$\mathcal{P}_k = \left\{P\in\mathcal{P}(\Xi^k): \mathbb{E}_P[\Psi(\xi)]\in\Gamma\right\}.$$

Based on the set $\Xi^k=\{\xi^1,\cdots,\xi^k\}$, we have the corresponding Voronoi tessellation of $\Xi$,
\begin{align*}
\Xi_i:=\left\{\xi\in\Xi: \norm{\xi-\xi^i}=\min_{1\leq j\leq k}\norm{\xi-\xi^j} \right\}, ~  i=1,\cdots,k.
\end{align*}
 Obviously, $\Xi=\bigcup_{i=1}^k \Xi_i$ and $\mathrm{int}(\Xi_i)\cap\mathrm{int}(\Xi_j)=\emptyset$ for any $i\neq j$.

In this paper, we make a commonly employed Slater type assumption for ambiguity set \eqref{AS} as follows (see e.g. \cite{DY2010distributionally,LPX2019discrete}).

\begin{assumption}
\label{Assu4}
There exist $P_0\in\mathcal{P}(\Xi)$ and $\alpha>0$ such that
$$\mathbb{E}_{P_0}[\Psi(\xi)] + \alpha\mathbb{B}\subseteq \Gamma.$$
\end{assumption}

\begin{proposition}
\label{Prop1}
Under Assumption \ref{Assu4}, there exists $\bar{k}>0$ such that $\mathcal{P}_k$ is nonempty for any $k\ge \bar{k}.$
\end{proposition}

The proof of Proposition \ref{Prop1} is given in Appendix.

To measure the distance between two probability distributions (and thus two ambiguity sets), we use the well-known Wasserstein metric.

\begin{definition}[\cite{KR1958space}]\label{Def4}
Let $\mathcal{H}:=\{\hbar:\Xi\rightarrow\mathbb{R}: \abs{\hbar(\xi_1)- \hbar(\xi_2)} \leq \norm{\xi_1-\xi_2}\}$. The Wasserstein metric between $P,Q\in\mathcal{P}(\Xi)$ is defined as
$$\mathbb{D}_W(P,Q)=\sup_{\hbar\in \mathcal{H}} \abs{\mathbb{E}_P[\hbar(\xi)]- \mathbb{E}_Q[\hbar(\xi)]}.$$
\end{definition}

Notice that Definition \ref{Def4} gives the definition of Wasserstein metric by using the Kantorovich-Rubinstein theorem. Another definition of Wasserstein metric is based on the joint probability distributions with marginal distributions $P$ and $Q$, see e.g. \cite{V2003topics} for more details.

The deviation distance between $\mathcal{P},\mathcal{Q}\subseteq \mathcal{P}(\Xi)$ induced by Wasserstein metric is denoted by $\mathbb{D}_W(\mathcal{P}, \mathcal{Q})$, namely,
$$\mathbb{D}_W(\mathcal{P}, \mathcal{Q}) := \sup_{P\in\mathcal{P}}\inf_{Q\in \mathcal{Q}} \mathbb{D}_W(P,Q).$$

Obviously, $\mathcal{P}_k\subseteq \mathcal{P}$, which implies that the deviation distance between $\mathcal{P}_k$ and $\mathcal{P}$ is always zero, i.e. $\mathbb{D}_W(\mathcal{P}_k, \mathcal{P})=0$. To derive the estimation of $\mathbb{D}_W(\mathcal{P}, \mathcal{P}_k)$, we assume in the rest of this section that $\Xi$ is bounded and denote
\begin{equation}
\label{gs7}
\Delta:=\sup_{P\in\mathcal{P}(\Xi)}\mathbb{D}_W(P,P_0) \leq \diam(\Xi)<+\infty,
\end{equation}
where $P_0$ is defined in Assumption \ref{Assu4}.

The following Hoffman's type lemma is from \cite[Theorem 2]{LPX2019discrete}.

\begin{lemma}[\cite{LPX2019discrete}]
\label{Th1}
Let Assumption \ref{Assu4} hold. Then, for any $Q\in\mathcal{P}(\Xi)$,
\begin{equation*}
\mathbb{D}_W(Q,\mathcal{P})\leq \frac{\Delta}{\alpha} \d(\mathbb{E}_Q[\Psi(\xi)], \Gamma),
\end{equation*}
where $\alpha$ is defined in Assumption \ref{Assu4} and $\Delta$ is defined in \eqref{gs7}.
\end{lemma}

For a sample set $\Xi^k=\{\xi^1,\cdots,\xi^k\}$ and the resulting Voronoi tessellation $\Xi_1,\cdots,\Xi_k$, we call the following probability distribution
$$P_k:=\sum_{i=1}^k P(\Xi_i) \mathbf{1}_{\xi^i}$$
the Voronoi projection of $P$. Denote by
\begin{equation}
\label{beta}
\beta_k=\max_{\xi\in\Xi}\min_{1\leq i\leq k}\norm{\xi-\xi^i}
\end{equation}
the Hausdorff distance between $\Xi$ and $\Xi^k$.

The following lemma gives an estimation of Wasserstein distance between $P$ and its Voronoi projection.

\begin{lemma}[Lemma 4.9, \cite{PP2014multistage}]\label{Lem5}
Let $P\in\mathcal{P}(\Xi)$ and $P_k$ be the Voronoi projection of $P$. Then
$\mathbb{D}_W(P,P_k)\leq \beta_k$.
\end{lemma}

Then, we have an estimation of $\mathbb{D}_W(\mathcal{P},\mathcal{P}_k)$ from \cite{LPX2019discrete}.

\begin{proposition}[\cite{LPX2019discrete}]
\label{Prop3}
Suppose that: (i) Assumption \ref{Assu4} holds; (ii) $\Psi$ is Lipschitz continuous with Lipschitz modulus $L_\Psi$; (iii) the sample set $\Xi^k$ satisfies $\beta_k\rightarrow 0$ as $k\rightarrow \infty$. Then, for sufficiently large $k$, we have
$$\mathbb{D}_W(\mathcal{P},\mathcal{P}_k) \leq L\beta_k,$$
where $L= 1+ \frac{2\diam(\Xi) L_\Psi}{\alpha}$ and $\alpha$ is defined in Assumption \ref{Assu4}.
\end{proposition}

In the remaining paper, we tacitly assume that both $\mathcal{P}$ and $\mathcal{P}_k$ are nonempty.

\section{Optimal values and solutions of  problems (\ref{DREC}) and (\ref{DREC-1})}\label{Sec3}

Problem (\ref{DREC}) can be rewritten as
\begin{equation*}
\min_{(x,y)\in {\cal D}}\max_{P\in {\mathcal P}}  ~g(x, y, P):=\theta(x) + h(\mathbb{E}_P[f(x,y(\xi),\xi)]),
\end{equation*}
where  $ {\mathcal D}=\{ (x, y) :  x\in X,  y(\xi) \in {\rm SOL}(M(\xi),  q(x,\xi)),  \, \xi \in \Xi\}.$

Note that the function $g(\cdot,\cdot, P)$ is not convex for a fixed $P\in {\mathcal P}$,  and $g(x,y,\cdot)$ is not concave for a fixed tuple $(x, y)\in {\mathcal D}$.
This usually implies that
$$ \min_{(x,y) \in {\mathcal D}}\max_{P\in {\mathcal P}}  g(x, y, P) \neq
\max_{P\in {\mathcal P}} \min_{(x,y)\in {\mathcal D}} g(x, y, P).
$$
By the saddle point existence theorem \cite{Debreu}, $g$ does not have a saddle point $(x^*, y^*, P^*)\in {\mathcal D}\times{\mathcal P} $ that satisfies
$$ g(x^*, y^*, P) \le g(x^*, y^*, P^*) \le g(x, y, P^*)$$
for any $(x, y, P)\in {\mathcal D}\times{\mathcal P}$.
Hence, we define global and local minimax points of problem \eqref{DREC} by using the idea in \cite{JNJ2020local}.

\begin{definition}\label{minimax}
 We call $(x^*, y^*,P^*) \in {\mathcal D}\times{\mathcal P}$ a \textbf{global minimax point} of problem \eqref{DREC}, if it satisfies
\begin{align*}
g(x^*, y^*, P) \le g(x^*, y^*, P^*) \leq \max_{Q\in \mathcal{P}}g(x,y,Q)
\end{align*}
for any $(x, y)\in \mathcal{D}$ and $P\in \mathcal{P}$.

 We call $(x^*, y^*,P^*) \in {\mathcal D}\times{\mathcal P}$ a \textbf{local minimax point} of problem \eqref{DREC}, if there exist a $\delta_0>0$ and a function $\varsigma:\mathbb{R}_+\rightarrow \mathbb{R}_+$ satisfying $\varsigma(\delta)\rightarrow 0$ as $\delta\rightarrow 0$, such that
\begin{align*}
g(x^*, y^*, P) \le g(x^*, y^*, P^*) \leq \max_{Q\in \mathcal{P},\mathbb{D}_W(Q,P^*)\leq \varsigma(\delta)}g(x,y,Q),
\end{align*}
for any $\delta\in (0,\delta_0]$, $(x, y)\in \mathcal{D}$
satisfying $\norm{x-x^*}\leq \delta$ and $\norm{y-y^*}_{\mathcal{L}_2} \leq \delta$ and $P\in \mathcal{P}$ satisfying $\mathbb{D}_W(P,P^*)\leq \delta$.
\end{definition}
Similarly, we can define global minimax points and local minimax points of problems \eqref{R-DREC}, \eqref{DDREC}
and \eqref{DREC-1}, respectively.

Using the notation $\mathcal{D}$, problem \eqref{DREC} regarding minimizers can be rewritten as
\begin{equation}
\label{DREC-Ref}
\min_{(x,y)\in {\cal D}} \Phi(x,y).
\end{equation}
By substituting the unique solution of LCP$(M(\xi)+\epsilon I, q(x,\xi))$ (denoted by $\hat{y}_\epsilon(x,\xi)$) into the objective function, problem \eqref{R-DREC} can be rewritten as
\begin{equation}
\label{RDREC}
\min\limits_{x\in X}~ \Phi_\epsilon (x,\hat{y}_\epsilon(x,\cdot)).
\end{equation}
Then its discrete approximation problem \eqref{DREC-1} can be rewritten as
\begin{equation}
\label{SSA-RDREC}
\min_{x\in X}~ \Phi_{\epsilon,k} (x,\hat{\mathbf{y}}_\epsilon(x)),
\end{equation}
where $\hat{\mathbf{y}}_\epsilon(x):=(\hat{y}_\epsilon(x,\xi^1)^\top,\cdots,\hat{y}_\epsilon(x,\xi^k)^\top)^\top$.

Denote by $\mathcal{D}^*$, $X_\epsilon^*$, $X_{\epsilon,k}^*$, and $v^*$, $v_\epsilon^*$, $v_{\epsilon,k}^*$ the sets of minimizers and optimal values of problems \eqref{DREC-Ref}, \eqref{RDREC} and \eqref{SSA-RDREC}, respectively. Let $X^*=\proj_x \mathcal{D}^*$ be the projection of $\mathcal{D}^*$ onto $x$-space.
In this section, we provide conditions for the existence of minimizers of problems \eqref{DREC}, \eqref{R-DREC} and \eqref{DREC-1}, and prove the convergence of optimal values and optimal solution sets of problem \eqref{DREC-1} to those of problem \eqref{DREC} as $\epsilon\downarrow0$ and $k\rightarrow \infty$.  The convergence analysis is divided  into two parts: the convergence of \eqref{R-DREC} to \eqref{DREC} as $\epsilon\downarrow0$ and the convergence of  \eqref{DREC-1} to \eqref{R-DREC} as $k\rightarrow \infty$  for a fixed $\epsilon>0$.

\subsection{Existence of solutions}
In this subsection, we provide some sufficient conditions for the existence of global minimax points of problems \eqref{DREC}, \eqref{R-DREC} and \eqref{DREC-1}, respectively.

Let $L_f(\xi^i)$ be the Lipschitz modulus of $f(\cdot,\cdot,\xi^i)$ and $L_\epsilon(\xi^i)$ be the Lipschitz modulus of $\hat{y}_\epsilon(\cdot,\xi^i)$.
For a fixed $\epsilon>0$ and $k>0$, we have
\begin{align*}
&~~~~\abs{ \max_{p\in\mathcal{P}_k}h( F(x,\hat{\mathbf{y}}_\epsilon(x))p ) - \max_{p\in\mathcal{P}_k}h( F(x',\hat{\mathbf{y}}_\epsilon(x'))p ) } \\
& \leq  \max_{p\in\mathcal{P}_k} \abs{ h( F(x,\hat{\mathbf{y}}_\epsilon(x))p ) - h( F(x',\hat{\mathbf{y}}_\epsilon(x'))p ) } \\
& \leq  L_h \max_{p\in\mathcal{P}_k} \norm{ \sum_{i=1}^k p_i f(x,\hat{y}_\epsilon(x,\xi^i),\xi^i) - \sum_{i=1}^k p_i f(x',\hat{y}_\epsilon(x',\xi^i),\xi^i) } \\
& \leq  L_h \max_{p\in\mathcal{P}_k} \sum_{i=1}^k p_i \norm{ f(x,\hat{y}_\epsilon(x,\xi^i),\xi^i) - f(x',\hat{y}_\epsilon(x',\xi^i),\xi^i) } \\
& \leq  L_h \max_{p\in\mathcal{P}_k} \sum_{i=1}^k p_i L_f(\xi^i)(\norm{x-x'} + L_\epsilon(\xi^i)\norm{x-x'} )\\
& \leq  L_h  \max_{1\leq i\leq k}L_f(\xi^i)\left( 1 + \max_{1\leq i\leq k} L_\epsilon(\xi^i) \right) \norm{x-x'}
\end{align*}
 for any $x,x'\in X$. Therefore, $\Phi_{\epsilon,k} (x,\hat{\mathbf{y}}_\epsilon(x))$  is Lipschitz continuous over the compact and convex set $X$, and  it has a minimizer $x^* \in X$.  Moreover,
 $\mathcal{P}_k$ is closed and bounded due to the continuity of $\Psi$ and $h( F(x^*,\hat{\mathbf{y}}_\epsilon(x^*))p )$ is continuous w.r.t. $p$. Hence, there is a maximizer $p^*$ such that $(x^*,p^*)$ is a global minimax point of problem \eqref{DREC-1}.

In what follows, we provide sufficient conditions for the existence of solutions to problems
\eqref{DREC} and \eqref{R-DREC}.

Recall that a sequentially compact set means that any sequence contained in this set has a convergent subsequence. We say a sequence of probability distribution $\{P_j\}\subseteq \mathcal{P}(\Xi)$ weakly converges to $P\in \mathcal{P}(\Xi)$ if $\mathbb{E}_{P_j}[g(\xi)]\rightarrow \mathbb{E}_{P}[g(\xi)]$ as $j\rightarrow \infty$ for any continuous and bounded function $g:\Xi\rightarrow \mathbb{R}$. A subset of $\mathcal{P}(\Xi)$ is weakly compact if any sequence in this subset has a weak convergent subsequence with a weak limit point contained in this subset.

\begin{assumption}\label{Sol}~
\begin{enumerate}[(i)]
\item There exists $\kappa:\Xi\rightarrow \mathbb{R}_+$ with $\max_{P\in\mathcal{P}} \mathbb{E}_P[\kappa(\xi)^2]< \infty$ such that
$$ \max_{(x,y) \in {\cal D}}  \norm{y(\xi)} < \kappa(\xi), \quad \xi\in \Xi.$$

\item There exists a tuple $(x_0, y_0) \in {\mathcal D}$  such that $\Phi(x_0,y_0)<\infty$.
\end{enumerate}

\end{assumption}

\begin{theorem}\label{Th-E}
Under Assumption \ref{Sol}, if $\mathcal{D}$ is a sequentially compact set, then the objective function of problem \eqref{DREC} satisfies $|\Phi(x,y)| <\infty$ for any $(x, y) \in {\cal D}$ and  problem \eqref{DREC} has a minimizer $(x^*,y^*)$.
\end{theorem}
\begin{proof}
By Assumption \ref{Sol}, we have
\begin{align*}
&|\Phi(x,y)-\Phi(x_0,y_0)| \\
& \le | \theta(x)-\theta(x_0)| +
 \max_{P\in\mathcal{P}}\abs{ \mathbb{E}_P[\norm{f(x,y(\xi),\xi)}] - \mathbb{E}_P[\norm{f(x_0,y_0(\xi),\xi)}]}\\
&\leq L_\theta\diam(X)  + L_h\max_{P\in\mathcal{P}}\mathbb{E}_P\left[ L_f(\xi)(\diam(X) + 2\kappa(\xi))\right]<\infty.
\end{align*}
Hence $|\Phi(x,y)| <\infty$ for any $(x, y) \in {\cal D}$, which means there is $v^*\in \mathbb{R}$ such that $v^*=\inf_{(x, y) \in {\cal D} } \Phi(x,y)$.
To prove that problem \eqref{DREC} has a minimizer $(x^*,y^*)$ under Assumption \ref{Sol}, we only need to show that $\Phi$ is lower semicontinuous (lsc) due to the sequential compactness of $\mathcal{D}$ \cite[Theorem 1.9]{RW2009variational}.

First, we show that for a fixed $x\in X$, $\max_{P\in {\cal P}} h(\mathbb{E}_P[f(x,\cdot,\xi)])$ is lsc.   By the Lipschitz continuity of $h$ and $f$, we know that
\begin{align*}
& \abs{ h(\mathbb{E}_P[f(x,y'(\xi),\xi)]) - h(\mathbb{E}_P[f(x,y(\xi),\xi)]) } \\
 &\leq L_h \mathbb{E}_P\left[L_f(\xi) \norm{y'(\xi) - y(\xi)}\right] \\
 &\leq L_h \left(\mathbb{E}_P[L_f^2(\xi)]\right)^{\frac{1}{2}}
 \left(\mathbb{E}_P[\norm{y'(\xi)-y(\xi)}^2]\right)^{\frac{1}{2}} \rightarrow0
\end{align*}
if $\left(\mathbb{E}_P[\norm{y'(\xi)-y(\xi)}^2]\right)^{\frac{1}{2}} \rightarrow0.$
Hence $h(\mathbb{E}_P[f(x,\cdot,\xi)])$ is continuous for fixed $x\in X$ and $P\in {\cal P}$, which implies that $\max_{P\in\mathcal{P}} h(\mathbb{E}_P[f(x,\cdot,\xi)])$ is lsc.

Similarly, we can show that $\max_{P\in {\cal P}} h(\mathbb{E}_P[f(\cdot,\cdot,\xi)])$ is lsc  over $\mathcal{D}$. Since $\theta$ is a Lipchitz continuous function of $x$, we derive that  $\Phi$ is lsc.
\qed
\end{proof}

\begin{theorem}\label{Th-E2}
Suppose that $\Xi$ is bounded.
Then for any $\epsilon>0$, problem \eqref{R-DREC} has a minimizer $(x^*,\hat{y}_\epsilon(x^*,\cdot))$. In addition, if  $\mathcal{P}$ is weakly compact,  then problem \eqref{R-DREC} has a global minimax point $(x^*,\hat{y}_\epsilon(x^*,\cdot),P^*)$.
\end{theorem}
\begin{proof}
For the existence of a minimizer $(x^*,\hat{y}_\epsilon(x^*,\cdot))$, we only need to show Assumption \ref{Sol} holds with $\Phi_\epsilon$ and
$${\cal D}_\epsilon=\{(x,y)  :  x\in X,\,\,  y(\xi) \in {\rm SOL}(M(\xi)+\epsilon I,  q(x,\xi)),  \, \xi \in \Xi\}$$
for any $\epsilon>0$.

Since $M(\xi)+\epsilon I$ is a positive definite matrix for any $\xi \in \Xi$, the solution set SOL$(M(\xi)+\epsilon I, q(x,\xi))$ has a unique vector $ \hat{y}_\epsilon(x, \xi)$ for any $x \in X$ and $\xi \in \Xi$.  Moreover, from Lipschitz continuity of $q$ and $M$ and boundness of $X$ and $\Xi$, there is a positive number $\Lambda_\epsilon >0$ such that
$$\Lambda_\epsilon \geq  \max_{x\in X, \xi \in \Xi} \max_{d\in [0,1]^n} \|(I-D+D(M(\xi)+\epsilon I))^{-1}\|\|\min(0, q(x,\xi))\|,$$
which implies $ \|\hat{y}_\epsilon(x, \xi)\|\le \Lambda_\epsilon $, for $x\in X, \xi \in \Xi$.
See error bounds for the LCP in \cite{CX2008perturbation,CPS1992linear}.
Hence, Assumption \ref{Sol} holds with $\Phi_\epsilon$ and
${\cal D}_\epsilon$.   Using a similar proof of Theorem \ref{Th-E}, we can show that problem \eqref{R-DREC} has a minimizer $(x^*,\hat{y}_\epsilon(x^*,\cdot))$.

Moreover, if $\mathcal{P}$ is weakly compact, due to the continuity and boundedness of $f(x^*,\hat{y}_\epsilon(x^*,\xi),\xi)$ w.r.t. $\xi$, the following maximization problem
\begin{equation*}
\max_{P\in\mathcal{P}}h(\mathbb{E}_P[f(x^*,\hat{y}_\epsilon(x^*,\xi),\xi)])
\end{equation*}
has a maximizer $P^*$. Hence $(x^*,\hat{y}_\epsilon(x^*,\cdot),P^*)$ is a global minimax point of problem \eqref{R-DREC}.
\qed
\end{proof}

\subsection{Convergence analysis between problems \eqref{DREC} and  \eqref{R-DREC}}

We use that LCP$(M(\xi)+\epsilon I, q(x,\xi))$ has a unique solution $ \hat{y}_\epsilon(x, \xi)$ to introduce the following auxiliary problem with fixed $\epsilon>0$ and $r_\epsilon >0$,
\begin{equation}\label{gs11}
\min\,  \Phi(x,y) \quad \mathrm{s.t.} \, x\in X,\,  y \in  \mathcal{M}_\epsilon(x):=\left\{y: \sup\limits_{P\in\mathcal{P}}\mathbb{E}_P[\norm{\hat{y}_\epsilon(x,\xi) - y(\xi)}^2]\leq r_\epsilon\right\}.
\end{equation}
Since $\hat{y}_\epsilon(x,\cdot)$ is a measurable function for any $x\in X$ and $X$ is bounded, $y\in \mathcal{M}_\epsilon(x) $ is a measurable function for any $x\in X$. By assumptions of Theorems \ref{Th-E}-\ref{Th-E2}, problems \eqref{DREC}, \eqref{R-DREC} and
 \eqref{gs11} have minimizers.  Denote by $\vartheta_\epsilon$ the optimal value of problem \eqref{gs11}. We make the following technical assumption.

\begin{assumption}\label{Assu3}
 $\abs{\vartheta_\epsilon - v^*} \to 0 ~\text{as}~\epsilon\downarrow0.$
\end{assumption}

\begin{remark}
Assumption \ref{Assu3} is a standard assumption and used widely in the perturbation analysis for parametric programming, see also \cite[Assumption 1]{CSW2015regularized} and \cite[Chapter 4]{BS2013perturbation}.
\end{remark}

\begin{theorem}
\label{Th2}
Under Assumption \ref{Assu3}, if  $r_\epsilon \rightarrow 0$ as $\epsilon\downarrow 0$, then we have
$$\abs{v_\epsilon^*-v^*}\rightarrow 0~\text{as}~\epsilon\downarrow0.$$
\end{theorem}

\begin{proof}
Using the Lipschitz continuity conditions on $\theta, h, f(\cdot,\cdot,\xi)$, we have
\begin{equation*}
\begin{aligned}
&\abs{ v_\epsilon^* - \vartheta_\epsilon } = \bigg\vert \min_{x\in X} \left(\theta(x) + \max_{P\in\mathcal{P}}h(\mathbb{E}_P[f(x,\hat{y}_\epsilon(x,\xi),\xi)]) \right) \\
&~~~~~~~~~~~~~~~~~-  \min_{x\in X} \left(\theta(x) + \min_{y\in \mathcal{M}_\epsilon(x)} \max_{P\in\mathcal{P}}h(\mathbb{E}_P[f(x,y(\xi),\xi)])  \right) \bigg\vert \\
&\leq \max_{x\in X} \left( \max_{P\in\mathcal{P}}h(\mathbb{E}_P[f(x,\hat{y}_\epsilon(x,\xi),\xi)])  -  \min_{y\in \mathcal{M}_\epsilon(x)} \max_{P\in\mathcal{P}}h(\mathbb{E}_P[f(x,y(\xi),\xi)]) \right) \\
&= \max_{x\in X} \max_{y\in \mathcal{M}_\epsilon(x)} \left(\max_{P\in\mathcal{P}}h(\mathbb{E}_P[f(x,\hat{y}_\epsilon(x,\xi),\xi)])  -   \max_{P\in\mathcal{P}}h(\mathbb{E}_P[f(x,y(\xi),\xi)]) \right)  \\
&\leq L_h\sup_{x\in X, y\in \mathcal{M}_\epsilon(x), P\in\mathcal{P}} \mathbb{E}_P[L_f(\xi)\|\hat{y}_\epsilon(x,\xi)-y(\xi)\|].
\end{aligned}
\end{equation*}
By H{\"o}lder inequality, we have that
\begin{align*}
\sup_{x\in X, y\in \mathcal{M}_\epsilon(x), P\in\mathcal{P} } \mathbb{E}_P[L_f(\xi)\|\hat{y}_\epsilon(x,\xi)-y(\xi)\|] &\leq \sqrt{r_\epsilon}  \left( \sup_{P\in\mathcal{P}}\mathbb{E}_P[L_f^2(\xi)] \right)^{\frac{1}{2}}
\rightarrow 0
\end{align*}
as $\epsilon\downarrow 0$ due to $\sup_{P\in\mathcal{P}}\mathbb{E}_P[L_f^2(\xi)]<\infty$ and $r_\epsilon\rightarrow 0$ as $\epsilon\downarrow 0$. Thus, we obtain
$\abs{v_\epsilon^* - \vartheta_\epsilon}\rightarrow 0$. This together with Assumption \ref{Assu3} implies that $\abs{v_\epsilon^*-v^*}\rightarrow 0$ as $\epsilon\downarrow 0$.
\qed
\end{proof}

Since $M(\xi)$ is positive semidefinite and the solution set of LCP$(M(\xi), q(x,\xi))$ is nonempty
for any $x\in X$ and $\xi \in \Xi$ by the assumption of relatively complete recourse, the unique solution    $\hat{y}_\epsilon(x,\xi)$ converges to
the least norm solution $\bar{y}(x,\xi)$ of  LCP$(M(\xi), q(x,\xi))$ as $\epsilon \downarrow 0$ \cite[Theorem 5.6.2]{CPS1992linear}, which is defined by
$$\bar{y}(x,\xi)=\argmin~ \{\norm{y(x,\xi)}: y(x,\xi) \in  {\rm SOL}(M(\xi), q(x,\xi))\}.$$

%

\begin{proposition}\label{Prop5}
Suppose that $\sup_{P\in\mathcal{P}}\mathbb{E}_P[ \norm{\hat{y}_\epsilon(x,\xi)-\bar{y}(x,\xi)}^2]\to 0$ uniformly as $\epsilon \downarrow 0$ in $ X$ and $v^*=\min_{x\in X} \left(\theta(x)+ \max_{P\in\mathcal{P}}h(\mathbb{E}_P[f(x,\bar{y}(x,\xi),\xi)])\right)$. Then $$\abs{v_\epsilon^*-v^*}\rightarrow 0~\text{and}~ \d(X_\epsilon^*,X^*)\rightarrow 0~\text{as}~\epsilon\downarrow0.$$
\end{proposition}
\begin{proof}
By assumptions and the boundness of $X$, we have
\begin{align*}
\abs{v_\epsilon^*-v^*}&=\bigg\vert \min_{x\in X} \left( \theta(x)+ \max_{P\in\mathcal{P}}h(\mathbb{E}_P[f(x,\hat{y}_\epsilon(x,\xi),\xi)]) \right) \\
&~~~~- \min_{x\in X} \left(\theta(x)+ \max_{P\in\mathcal{P}}h(\mathbb{E}_P[f(x,\bar{y}(x,\xi),\xi)])\right) \bigg\vert\\
&\leq \max_{x\in X}\abs{\max_{P\in\mathcal{P}}h(\mathbb{E}_P[f(x,\hat{y}_\epsilon(x,\xi),\xi)]) - \max_{P\in\mathcal{P}} h(\mathbb{E}_P[f(x,\bar{y}(x,\xi),\xi)])}\\
&\leq L_h \max_{x\in X}\max_{P\in\mathcal{P}}\mathbb{E}_P[ \norm{f(x,\hat{y}_\epsilon(x,\xi),\xi)] - \mathbb{E}_P[f(x,\bar{y}(x,\xi),\xi)}]\\
&\leq L_h \max_{x\in X}\max_{P\in\mathcal{P}}\mathbb{E}_P[ L_f(\xi)\norm{\hat{y}_\epsilon(x,\xi)-\bar{y}(x,\xi)}] \\
&\leq L_h \max_{x\in X}   \left( \sup_{P\in\mathcal{P}}\mathbb{E}_P[L_f^2(\xi)] \right)^{\frac{1}{2}}\left( \sup_{P\in\mathcal{P}}\mathbb{E}_P[ \norm{\hat{y}_\epsilon(x,\xi)-\bar{y}(x,\xi)}^2] \right)^{\frac{1}{2}}
\rightarrow 0
\end{align*}
as $\epsilon\downarrow0$. Moreover, from the above inequalities, we find that
$\Phi_\epsilon(x, \hat{y}_\epsilon(x,\cdot))$ converges to $\Phi(x, \bar{y}(x,\cdot))$ uniformly w.r.t. $x$ over $X$ as
$\epsilon \downarrow 0$. By \cite[Theorem 5.3]{SDR2014lectures}, we derive
the convergence of $\d(X_\epsilon^*,X^*)$.
\qed
\end{proof}

\subsection{Convergence analysis between problems \eqref{R-DREC} and \eqref{DREC-1}}

In this subsection, for a fixed $\epsilon>0$, we consider the convergence between problems \eqref{R-DREC} and \eqref{DREC-1} as $k\rightarrow\infty$. Let the support set $\Xi$ be bounded throughout this subsection.

By \cite[Theorem 2.8]{CX2008perturbation}, there exists an $\alpha>0$ such that for any $\xi_1, \xi_2\in \Xi$,
$$\max_{x\in X}\|\hat{y}_\epsilon(x,\xi_1)-\hat{y}_\epsilon(x, \xi_2)\|\le \frac{\alpha}{\epsilon}\max_{x\in X,\xi \in \Xi} \|q(x,\xi)\|\|\xi_1-\xi_2\|=:\hat{L}\|\xi_1-\xi_2\|,$$
where the existence of a constant $ \hat{L}>0$ employs the boundness of $X$ and $\Xi$ and Lipschitz continuity of $q$.
Moreover, if $f$ is Lipschitz continuous, there exists an $\bar{L}_f>0$ such that for any $x\in X$,
\begin{equation}
\label{gs18}
\norm{f(x,\hat{y}_\epsilon(x,\xi_1),\xi_1)-f(x,\hat{y}_\epsilon(x,\xi_2),\xi_2)}\leq \bar{L}_f \norm{\xi_1-\xi_2}.
\end{equation}

We give the following quantitative stability results between problems \eqref{RDREC} and \eqref{SSA-RDREC} based on Wasserstein metric.

\begin{theorem}
\label{Th4}
Let $f$ be Lipschitz continuous. There exists an $L>0$ such that
\begin{align}
\abs{v_{\epsilon,k}^* - v_\epsilon^* }&\leq L \mathbb{D}_W(\mathcal{P}, \mathcal{P}_k),\label{Th4_1}\\
\d(X_{\epsilon,k}^*, X_\epsilon^*) &\leq \mathcal{R}^{-1}(2L \mathbb{D}_W(\mathcal{P}, \mathcal{P}_k)),\label{Th4_2}
\end{align}
where $\mathcal{R}:\mathbb{R}_+\rightarrow \mathbb{R}_+$ is the growth function of problem \eqref{R-DREC}, i.e.,
\begin{equation*}
\mathcal{R}(\tau) :=\min \left\{\theta(x) + \max_{P\in\mathcal{P}}h(\mathbb{E}_P[f(x,\hat{y}_\epsilon(x,\xi),\xi)]) - v_\epsilon^*: \d(x, X_\epsilon^*)\geq \tau, x\in X \right\}
\end{equation*}
and its inverse
\begin{equation*}
\mathcal{R}^{-1}(t):=\sup\{\tau\in \mathbb{R}_+: \mathcal{R}(\tau) \leq t\}.
\end{equation*}
\end{theorem}
\begin{proof}
Using (\ref{gs18}), we have
\begin{align*}
\abs{ v_\epsilon^* - v_{\epsilon,k}^*} &=  \bigg\vert \min_{x\in X} \left( \theta(x) + \max_{P\in\mathcal{P}}h(\mathbb{E}_P[f(x,\hat{y}_\epsilon(x,\xi),\xi)]) \right)  \\
&~~~~~~~~~~~~~~~~~~- \min_{x\in X} \left(\theta(x) + \max_{Q\in\mathcal{P}_k}h(\mathbb{E}_Q[f(x,\hat{y}_\epsilon(x,\xi),\xi)]) \right) \bigg\vert \\
&\leq\max_{x\in X} \abs{\max_{P\in\mathcal{P}}h(\mathbb{E}_P[f(x,\hat{y}_\epsilon(x,\xi),\xi)])  - \max_{Q\in\mathcal{P}_k}h(\mathbb{E}_Q[f(x,\hat{y}_\epsilon(x,\xi),\xi)]) }\\
&=\max_{x\in X}  \left( \max_{P\in\mathcal{P}}h(\mathbb{E}_P[f(x,\hat{y}_\epsilon(x,\xi),\xi)])  - \max_{Q\in\mathcal{P}_k}h(\mathbb{E}_Q[f(x,\hat{y}_\epsilon(x,\xi),\xi)]) \right)\\
&=\max_{x\in X}\max_{P\in\mathcal{P}} \min_{Q\in\mathcal{P}_k} \left( h(\mathbb{E}_P[f(x,\hat{y}_\epsilon(x,\xi),\xi)]) - h(\mathbb{E}_Q[f(x,\hat{y}_\epsilon(x,\xi),\xi)] \right)\\
&\leq \max_{x\in X}L_h \max_{P\in\mathcal{P}} \min_{Q\in\mathcal{P}_k} \norm{\mathbb{E}_P[f(x,\hat{y}_\epsilon(x,\xi),\xi)] - \mathbb{E}_Q[f(x,\hat{y}_\epsilon(x,\xi),\xi)] }\\
&\leq L_h\bar{L}_f \mathbb{D}_W(\mathcal{P},\mathcal{P}_k),
\end{align*}
where  the second equality follows from  $\mathcal{P}_k\subseteq \mathcal{P}$,
and the last inequality follows from Definition \ref{Def4}.  We obtain (\ref{Th4_1}) by setting  $L:=L_h\bar{L}_f$.

From (\ref{Th4_1}), for any $\tilde{x}\in X_{\epsilon,k}^*$, it holds that
\begin{align*}
2L \mathbb{D}_W(\mathcal{P}, \mathcal{P}_k) &\geq L \mathbb{D}_W(\mathcal{P}, \mathcal{P}_k ) + \abs{v_\epsilon^* - v_{\epsilon,k}^*}\\
&\geq L \mathbb{D}_W(\mathcal{P}, \mathcal{P}_k ) + v_{\epsilon,k}^* - v_\epsilon^*\\
&\geq \theta(\tilde{x}) + \max_{P\in\mathcal{P}}h(\mathbb{E}_P[f(\tilde{x},\hat{y}_\epsilon(\tilde{x},\xi),\xi)])  \\
&~~~~- \left(\theta(\tilde{x}) + \max_{Q\in\mathcal{P}_k}h(\mathbb{E}_Q[f(\tilde{x},\hat{y}_\epsilon(\tilde{x},\xi),\xi)]) \right) + v_{\epsilon,k}^* - v_\epsilon^*\\
&= \theta(\tilde{x}) + \sup\limits_{P\in\mathcal{P}}h(\mathbb{E}_P[f(\tilde{x},\hat{y}_\epsilon(\tilde{x},\xi),\xi)]) - v_\epsilon^* \\
&\geq \mathcal{R}( \d(\tilde{x}, X_\epsilon^*) ),
\end{align*}
which implies that
$$\d(\tilde{x}, X_\epsilon^*) \leq \mathcal{R}^{-1}\left( 2L \mathbb{D}_W(\mathcal{P}, \mathcal{P}_k)\right).$$
Due to the arbitrariness of $\tilde{x}\in X_{\epsilon,k}^*$, we derive (\ref{Th4_2}).
\qed
\end{proof}

By directly using Theorem \ref{Th4} and Proposition \ref{Prop3}, we obtain the following convergence results.
\begin{theorem}
\label{Th3}
Under the conditions of Proposition \ref{Prop3} and Theorem \ref{Th4}, we have
$$\abs{v_{\epsilon,k}^*-v_\epsilon^*}\rightarrow 0 ~\text{and}~\d(X_{\epsilon,k}^*,X_\epsilon^*)\rightarrow 0~\text{as}~ k\rightarrow \infty.$$
\end{theorem}

Finally, Theorem \ref{Th2} together with Theorem \ref{Th3} entails the following convergence results
from \eqref{DREC-1} to \eqref{DREC}.

\begin{theorem} Suppose that the conditions of Theorems \ref{Th2} and \ref{Th3} hold. Then
$$\lim_{\epsilon\downarrow 0}\lim_{k\rightarrow \infty} v_{\epsilon,k}^* = v^*.$$
If, moreover, the conditions of Proposition \ref{Prop5} hold, then
$$\lim_{\epsilon\downarrow 0}\lim_{k\rightarrow \infty} \d(X_{\epsilon,k}^*,X^*)=0.$$
\end{theorem}

\section{Stationarity of problems (\ref{DREC}) and (\ref{DREC-1})}\label{Sec4}

In this section, we consider the convergence of stationary points of problem (\ref{DREC-1}) to these of problem (\ref{DREC}) as $\epsilon\downarrow 0$ and $k\rightarrow \infty$. We first consider the stationary points defined by the regular normal cone in \cite{RW2009variational}.

The regular normal cone to a closed set $\Omega$ at $\bar{x}\in \Omega$, denoted by $\widehat{\mathcal{N}}_\Omega(\bar{x})$, is
$$\widehat{\mathcal{N}}_\Omega(\bar{x}):=\{d: d^\top (x-\bar{x})\leq o(\norm{x-\bar{x}}),~ \forall x\in \Omega\},$$
where $o(\cdot)$ means that $o(a)/a\rightarrow0$ as $a\downarrow0$. The (limiting) normal cone to a closed set $\Omega$ at $\bar{x}\in \Omega$, denoted by $\mathcal{N}_\Omega(\bar{x})$, is
$$\mathcal{N}_\Omega(\bar{x}):=\{d: \exists\, x^k \in \Omega, x^k\rightarrow \bar{x},  \,\, d^k\rightarrow d \, {\rm with} \,  d^k \in \widehat{\mathcal{N}}_\Omega(x^k)\}.$$

$\widehat{\mathcal{N}}_\Omega(\bar{x})$ is a closed and convex cone and $\widehat{\mathcal{N}}_\Omega(\cdot)$ is outer semicontinuous over $\Omega$. $\mathcal{N}_\Omega(\bar{x})$ is a closed cone. Generally, we have $\widehat{\mathcal{N}}_\Omega(\bar{x}) \subseteq \mathcal{N}_\Omega(\bar{x})$.
They are consistent when $\Omega$ is convex, namely,
$$\widehat{\mathcal{N}}_\Omega(\bar{x})=\mathcal{N}_\Omega(\bar{x})=\{d: d^\top (x-\bar{x})\leq 0,~ \forall x\in \Omega\}.$$

A necessary condition for $\bar{x}$ to be a local minimizer of $\min_{x\in\Omega}~  \phi(x)$, where $\phi$ is a differentiable function over $\Omega$, is (see \cite[Theorem 6.12]{RW2009variational})
$$0\in \nabla \phi(\bar{x}) + \widehat{\mathcal{N}}_\Omega(\bar{x}).$$

In Section \ref{Sec4-1}, we focus on  the concepts of stationary points of  problems \eqref{DREC}, \eqref{DDREC} and \eqref{DREC-1}. In Section \ref{Sec4-2}, we give convergence results regarding the stationary points between problems \eqref{DREC-1} and \eqref{DREC} as $\epsilon \downarrow 0$ and $k\to \infty$.

\subsection{Concepts of stationarity}\label{Sec4-1}

We first consider the concepts of stationary points of the discrete problems \eqref{DDREC} and \eqref{DREC-1}. To simplify the notation, we denote
\begin{align*}
&G(x,\mathbf{y},p)=\theta(x) + h( F(x,\mathbf{y})p ),\\
&\mathcal{Z}=\{(x,\mathbf{y}): x\in X, \,\, \mathbf{y}\in {\rm SOL}(\mathbf{M},\mathbf{q}(x))\},\\
&\mathcal{Z}(x)=\{\mathbf{y}:(x,\mathbf{y})\in \mathcal{Z}\}, \,\,\, ~~\mathcal{Z}(\mathbf{y})=\{x:(x,\mathbf{y})\in \mathcal{Z}\},\\
&\mathcal{Z}_\epsilon=\{(x,\mathbf{y}): x\in X, \,\, \mathbf{y}\in {\rm SOL}(\mathbf{M}+\epsilon I ,\mathbf{q}(x))\},\\
&\mathcal{Z}_\epsilon(x)=\{\mathbf{y}:(x,\mathbf{y})\in \mathcal{Z}_\epsilon\}, ~~ \,\,\, \mathcal{Z}_\epsilon(\mathbf{y})=\{x:(x,\mathbf{y})\in \mathcal{Z}_\epsilon\}.
\end{align*}
Then problems \eqref{DDREC} and \eqref{DREC-1} can be rewritten, respectively,  as
\begin{equation}
\label{D-DREC}
\min\limits_{(x,\mathbf{y})\in \mathcal{Z}} \max_{p\in\mathcal{P}_k}~ G(x,\mathbf{y},p)
\end{equation}
and
\begin{equation}
\label{DREC-4}
\min\limits_{(x,\mathbf{y})\in \mathcal{Z}_\epsilon} \max_{p\in\mathcal{P}_k}~  G(x,\mathbf{y},p).
\end{equation}
Note that the set ${\cal P}_k$ is convex and bounded. We have the following optimality condition for a local minimax point.

\begin{proposition}[optimality condition for a local minimax point]
\label{Prop2}
If $(\bar{x},\bar{\mathbf{y}},\bar{p})\in \mathcal{Z}\times \mathcal{P}_k$ is a local minimax point of problem \eqref{D-DREC}, then it satisfies
\begin{equation}\label{stationary0}
\begin{cases}
0\in \nabla_{(x,\mathbf{y})} G(\bar{x},\bar{\mathbf{y}},\bar{p}) + \widehat{\mathcal{N}}_{\mathcal{Z}}(\bar{x},\bar{\mathbf{y}}),\\
0\in -\nabla_{p}G(\bar{x},\bar{\mathbf{y}},\bar{p}) + \mathcal{N}_{\mathcal{P}_k}(\bar{p}).
\end{cases}
\end{equation}
\end{proposition}

\begin{proof}
Let $(\bar{x},\bar{\mathbf{y}},\bar{p})$ be a local minimax point. To simplify the notation, denote $\bar{z}=(\bar{x},\bar{\mathbf{y}})$. According to the definition of local minimax points, there exist a $\delta_0>0$ and a function $\varsigma:\mathbb{R}_+\rightarrow \mathbb{R}_+$ satisfying $\varsigma(\delta)\rightarrow 0$ as $\delta\rightarrow 0$, such that for any $\delta\in (0,\delta_0]$ and $(z,p)\in \mathcal{Z}\times \mathcal{P}_k$ satisfying $\norm{z-\bar{z}}\leq \delta$ and $\norm{p-\bar{p}}\leq \delta$, we have
\begin{equation}\label{stationary1}
G(\bar{z},p)\leq G(\bar{z},\bar{p}) \leq \max_{p'\in \{p\in \mathcal{P}_k:  \norm{p-\bar{p}}\leq \varsigma(\delta)\}} G(z,p').
\end{equation}
Obviously, the first inequality in \eqref{stationary1} implies  the second assertion in \eqref{stationary0}.

Next, we verify the first assertion in \eqref{stationary0}.
For any $z\in \mathcal{Z}$, let $\mathcal{T}_\mathcal{Z}(z)$ be the tangent cone of $\mathcal{Z}$ at $z$ (see \cite[Definition 6.1]{RW2009variational}) and
$$\tilde{p}(z)\in \argmax_{p'\in \{p\in \mathcal{P}_k:  \norm{p-\bar{p}}\leq \varsigma(\delta)\}} G(z,p').$$

From the second inequality in (\ref{stationary1}), we have, for any $w\in\mathcal{T}_\mathcal{Z}(z)$ with $\norm{w}= \delta$ and $\bar{z}+w \in \mathcal{Z}$, that
\begin{align*}
0 &\leq G(\bar{z}+w,\tilde{p}(\bar{z}+w)) - G(\bar{z},\bar{p}) \\
&\leq G(\bar{z}+w,\tilde{p}(\bar{z}+w)) -  G(\bar{z},\tilde{p}(\bar{z}+w)) + G(\bar{z},\tilde{p}(\bar{z}+w)) - G(\bar{z},\bar{p}) \\
&\leq G(\bar{z}+w,\tilde{p}(\bar{z}+w)) -  G(\bar{z},\tilde{p}(\bar{z}+w))\\
&= \nabla_z G(\bar{z},\tilde{p}(\bar{z}+w))^\top w  + o(\norm{w})\\
&= \nabla_z G(\bar{z},\bar{p})^\top w + (\nabla_z G(\bar{z},\tilde{p}(\bar{z}+w)) - \nabla_z G(\bar{z},\bar{p}) )^\top w  + o(\norm{w})\\
&= \nabla_z G(\bar{z},\bar{p})^\top w  + o(\norm{w}),
\end{align*}
where the last equality follows from $\norm{\tilde{p}(\bar{z}+w) - \bar{p}}\leq \varsigma(\delta)\rightarrow 0$ as $\delta\to 0$, which implies
$$(\nabla_z G(\bar{z},\tilde{p}(\bar{z}+w)) - \nabla_z G(\bar{z},\bar{p}) )^\top w = o(\norm{w}).$$
Thus, $\nabla_z G(\bar{z},\bar{p})^\top w \geq 0$ for any $w\in\mathcal{T}_\mathcal{Z}(z)$, which indicates (see \cite[Proposition 6.5]{RW2009variational}) the first assertion in \eqref{stationary0}.
\qed
\end{proof}

It is noteworthy that the necessary condition for local minimax points in Proposition \ref{Prop2} is also a necessary condition for saddle points, see e.g. \cite{RHLNSH2020non,PS2011nonconvex,NSHLR2019solving}. $\mathcal{Z}(\bar{x})$ is a convex set because of the positive semidefiniteness of $\mathbf{M}.$
If we consider $(x,\mathbf{y})$ individually, it then leads to the concept of block coordinatewise stationarity (see e.g. \cite{XY2013block}) as follows:
\begin{equation}
\label{BS}
\begin{cases}
0\in \nabla_{x} G(\bar{x},\bar{\mathbf{y}},\bar{p}) + \widehat{\mathcal{N}}_{\mathcal{Z}(\bar{\mathbf{y}})}(\bar{x}),\\
0\in \nabla_{\mathbf{y}} G(\bar{x},\bar{\mathbf{y}},\bar{p}) + \mathcal{N}_{\mathcal{Z}(\bar{x})}(\bar{\mathbf{y}}),\\
0\in -\nabla_{p}G(\bar{x},\bar{\mathbf{y}},\bar{p}) + \mathcal{N}_{\mathcal{P}_k}(\bar{p}).
\end{cases}
\end{equation}

Condition \eqref{BS} is a weaker necessary condition than \eqref{stationary0} for local optimality of
problem \eqref{D-DREC} (see e.g. \cite[Remark 2.2]{XY2013block}).
The second assertion in \eqref{BS} corresponds to a necessary condition for local optimality of the mathematical programming with linear complementarity constraint (MPLCC). Consider
\begin{equation}
\label{gs8}
\begin{array}{cl}
\min\limits_{\mathbf{y}}& G(\bar{x},\mathbf{y},\bar{p}) \\
\mathrm{s.t.} &0\leq \mathbf{y}  \bot (\mathbf{M}+\epsilon \mathbf{I} ) \mathbf{y}  + \mathbf{q}(\bar{x}) \geq 0.
\end{array}
\end{equation}

To characterize optimality conditions of MPLCC, we define the following index sets:
\begin{align*}
\mathcal{I}_{+0}(\mathbf{y})&=\{i:\mathbf{y}_i>0,\,\,  ((\mathbf{M}+\epsilon \mathbf{I})  \mathbf{y}  + \mathbf{q}(\bar{x}))_i=0\},\\
\mathcal{I}_{0+}(\mathbf{y})&=\{i:\mathbf{y}_i=0,\, \,  ((\mathbf{M}+\epsilon \mathbf{I})  \mathbf{y}  + \mathbf{q}(\bar{x}))_i>0\},\\
\mathcal{I}_{00}(\mathbf{y})&=\{i:\mathbf{y}_i=0,\,\,  ((\mathbf{M}+\epsilon \mathbf{I})  \mathbf{y}  + \mathbf{q}(\bar{x}))_i=0\}.
\end{align*}

\begin{definition}[stationary points for MPLCC, \cite{CY2009class}] \label{Def5}
(i) We say that $\mathbf{y}^*$ is a (weak stationary) W-stationary point of problem \eqref{gs8} if there exist multipliers $(\lambda,\mu)\in\mathbb{R}^{mk}\times\mathbb{R}^{mk}$ such that
\begin{equation}
\label{gs9}
\begin{aligned}
&\nabla_\mathbf{y} G(\bar{x},\mathbf{y}^*,\bar{p}) - \lambda - (\mathbf{M}^\top+\epsilon \mathbf{I}) \mu = 0,\\
& \lambda_i =0 ~\text{for}~ i\in \mathcal{I}_{+0}(\mathbf{y}^*),~ \mu_i =0 ~\text{for}~ i\in \mathcal{I}_{0+}(\mathbf{y}^*).
\end{aligned}
\end{equation}

(ii) We say that $\mathbf{y}^*$ is a (Clarke stationary) C-stationary point of problem \eqref{gs8} if there exist $(\lambda,\mu)\in\mathbb{R}^{mk}\times\mathbb{R}^{mk}$ satisfying \eqref{gs9} and
$$\lambda_i\mu_i\geq 0~\text{for}~ i\in \mathcal{I}_{00}(\mathbf{y}^*).$$

(iii) We say that $\mathbf{y}^*$ is an (Mordukhovich stationary) M-stationary point of problem \eqref{gs8} if there exist $(\lambda,\mu)\in\mathbb{R}^{mk}\times\mathbb{R}^{mk}$ satisfying \eqref{gs9} and
$$\lambda_i>0,~\mu_i>0~\text{or}~ \lambda_i\mu_i = 0 ~\text{for}~ i\in \mathcal{I}_{00}(\mathbf{y}^*).$$

(iv) We say that $\mathbf{y}^*$ is an (strong stationary) S-stationary point of problem \eqref{gs8} if there exist $(\lambda,\mu)\in\mathbb{R}^{mk}\times\mathbb{R}^{mk}$ satisfying \eqref{gs9} and
$$\lambda_i\geq 0,~\mu_i\geq 0~\text{for}~ i\in \mathcal{I}_{00}(\mathbf{y}^*).$$
\end{definition}

Obviously, we have the following observation:
$$\text{S-stationarity} \Rightarrow \text{M-stationarity} \Rightarrow \text{C-stationarity} \Rightarrow \text{W-stationarity}.$$
For more details about the optimality conditions of the mathematical programming with equilibrium constraints (MPEC), we refer to monograph \cite[Chapter 3]{LPR1996mathematical}.

Combined \eqref{BS} with Definition \ref{Def5}, we give the following concepts of block coordinatewise stationary points of problem \eqref{DDREC}.

\begin{definition}[block coordinatewise $\bullet$-stationary points of \eqref{DDREC}]\label{Def3}
Let $(\bar{x},\bar{\mathbf{y}},\bar{p})\in \mathcal{Z} \times \mathcal{P}_k$. We say it is a block coordinatewise $\bullet$-stationary point of problem \eqref{D-DREC} if it satisfies
\begin{equation*}
\begin{cases}
&0\in \nabla_x G(\bar{x},\bar{\mathbf{y}},\bar{p})+ \widehat{\mathcal{N}}_{\mathcal{Z}(\bar{\mathbf{y}})}(\bar{x}),\\
& \bar{\mathbf{y}} ~\text{is a $\bullet$-stationary point of problem \eqref{gs8}} \, {\rm with}\, \epsilon=0,\\
&0\in -\nabla_{p}G(\bar{x},\bar{\mathbf{y}},\bar{p}) + \mathcal{N}_{\mathcal{P}_k}(\bar{p}),
\end{cases}
\end{equation*}
where ``$\bullet$" can be W, C, M and S.
\end{definition}

Analogously, we give the block coordinatewise stationary point of regularized problem \eqref{DREC-1}.
Since there is a unique solution of LCP$(\mathbf{M}+\epsilon I, \mathbf{q}(\bar{x}))$ for any $\bar{x}\in X$, we use the following definition of block coordinatewise stationary points for problem \eqref{DREC-1}.
\begin{definition}[block coordinatewise stationary points of \eqref{DREC-1}]\label{WSP}
We call $(\bar{x},\bar{\mathbf{y}},\bar{p})\in \mathcal{Z}_\epsilon \times \mathcal{P}_k$ a block coordinatewise stationary point of problem \eqref{DREC-4}, if it satisfies
\begin{equation}
\label{WS}
\begin{cases}
0\in \nabla_x G(\bar{x},\bar{\mathbf{y}},\bar{p})+ \widehat{\mathcal{N}}_{\mathcal{Z}_\epsilon(\bar{\mathbf{y}})}(\bar{x}),\\
\bar{\mathbf{y}} \in  \mathrm{SOL}(\mathbf{M}+\epsilon I, \mathbf{q}(\bar{x})),\\
0\in -\nabla_{p}G(\bar{x},\bar{\mathbf{y}},\bar{p}) + \mathcal{N}_{\mathcal{P}_k}(\bar{p}).
\end{cases}
\end{equation}
\end{definition}

The following lemma shows that a point satisfying \eqref{WS} must be a block coordinatewise S-stationary point of problem \eqref{DREC-1}.

\begin{lemma}\label{Lem3}
For fixed $\bar{x}\in X$ and $\bar{p}\in\mathcal{P}_k$,
$\hat{\mathbf{y}}_\epsilon(\bar{x})$ is an S-stationary point of the MPLCC \eqref{gs8}.
\end{lemma}
\begin{proof} The feasible set of \eqref{gs8} with $\epsilon>0$ has a unique vector $\hat{\mathbf{y}}_\epsilon(\bar{x})$.
From \cite[Proposition 2.2, (ii) and (iii)]{GL2013notes}, $\hat{\mathbf{y}}_\epsilon(\bar{x})$ is an S-stationary point since
 both $\mathbf{y}$ and $({\mathbf{M}} +\epsilon \mathbf{I})\mathbf{y}+q(\bar{x})$ are linear functions of $\mathbf{y}$.
\qed
\end{proof}

\begin{remark}\label{Rem1}
Note that $\mathcal{Z}_\epsilon(\bar{\mathbf{y}})\subseteq X$ and thus $\mathcal{N}_{X}(\bar{x})\subseteq \widehat{\mathcal{N}}_{\mathcal{Z}_\epsilon(\bar{\mathbf{y}})}(\bar{x})$. Hence, if $(\bar{x},\bar{\mathbf{y}},\bar{p})\in \mathcal{Z}_\epsilon \times \mathcal{P}_k$ satisfies
\begin{equation}
\label{gs17}
\begin{cases}
0\in \nabla_x G(\bar{x},\bar{\mathbf{y}},\bar{p})+ \mathcal{N}_{X}(\bar{x}),\\
\bar{\mathbf{y}}\in  \mathrm{SOL}(\mathbf{M}+\epsilon I, \mathbf{q}(\bar{x})),\\
0\in -\nabla_{p}G(\bar{x},\bar{\mathbf{y}},\bar{p}) + \mathcal{N}_{\mathcal{P}_k}(\bar{p}),
\end{cases}
\end{equation}
then $(\bar{x},\bar{\mathbf{y}},\bar{p})$ satisfies (\ref{WS}).
 The main considerations for \eqref{gs17} are that $\mathcal{Z}_\epsilon(\bar{\mathbf{y}})$ is not convex for given $\bar{\mathbf{y}}$ if $q$ is a nonlinear function, and $\mathcal{Z}_\epsilon(\cdot)$ is not Lipschitz continuous in the sense of Hausdorff distance, which can lead to failing the following convergence analysis (see e.g. \cite{XY2013block}). In view of these, we use in the sequel $\mathcal{N}_{X}(\bar{x})$ rather than $\widehat{\mathcal{N}}_{\mathcal{Z}_\epsilon(\bar{\mathbf{y}})}(\bar{x})$.
\end{remark}

MPEC is generally difficult to deal with because its constraints fail to satisfy the standard Mangasarian-Fromovitz constraint qualification (originated from \cite{MF1967fritz}) at any feasible point \cite[Proposition 1.1]{YZZ1997exact}. We recall the definition of MPEC linear independence constraint qualification (MPEC-LICQ).

\begin{definition}[\cite{SS2000mathematical}]
We say that MPEC-LICQ holds at $\mathbf{y}$ for problem \eqref{gs8} if
$$\left\{\mathbf{e}_i : i\in \mathcal{I}_{0+}(\mathbf{y})\cup \mathcal{I}_{00}(\mathbf{y}) \right\}\cup \left\{\mathbf{M}_i^\top+\mathbf{e}_i :  i\in \mathcal{I}_{+0}(\mathbf{y})\cup \mathcal{I}_{00}(\mathbf{y}) \right\}$$
is linearly independent, where $\mathbf{e}_i$ is the $i$th column of the $mk\times mk$ identify matrix and  $\mathbf{M}_i$ is the $i$th row of the matrix $\mathbf{M}$.
\end{definition}

In what follows, we give the stationarity of problem \eqref{DREC}. To this end, we assume that for arbitrary probability distribution/measure $P\in \mathcal{P}$ there exists a corresponding density function $p:\Xi\rightarrow \mathbb{R}_+$ such that $P(\d\xi)=p(\xi)\d\xi$\footnote{We can generally assume that $P(\d\xi)=p(\xi)\mathbb{Q}(\d\xi)$ for some nominal probability distribution $\mathbb{Q}$. We know from Radon-Nikodym theorem (see e.g. \cite[Theorem 7.32]{SDR2014lectures}) that there exists such a density function $p(\xi)$ if and only if $P$ is absolutely continuous w.r.t. $\mathbb{Q}$. Here we neglect $\mathbb{Q}$ to simplify the notation.}. Denote by $\mathfrak{P}$ the collection of all density functions of $P\in\mathcal{P}$.

To define the block coordinatewise stationarity of problem \eqref{DREC}, we first give the definition of stationary points for stochastic MPEC (see Definition \ref{Def5} for the discrete version).

For fixed $(\bar{x},\bar{P})$, we consider the following stochastic MPEC problem
\begin{equation}
\label{MPEC-1}
\begin{array}{cl}
\min\limits_{y}& \theta(\bar{x}) + h(\mathbb{E}_{\bar{P}}[f(\bar{x},y(\xi),\xi)]) \\
\mathrm{s.t.}& 0\leq y(\xi) \bot M(\xi)y(\xi) + q(\bar{x},\xi) \geq 0~\text{for a.e.}~\xi\in \Xi.
\end{array}
\end{equation}

Define the following index sets:
\begin{equation}
\label{gs3}
\begin{aligned}
\mathcal{I}_{+0}(y;\xi)&=\{i:y_i(\xi)>0, (M(\xi)  y(\xi)  + q(\bar{x},\xi))_i=0\},\\
\mathcal{I}_{0+}(y;\xi)&=\{i:y_i(\xi)=0, (M(\xi)  y(\xi)  + q(\bar{x},\xi))_i>0\},\\
\mathcal{I}_{00}(y;\xi)&=\{i:y_i(\xi)=0, (M(\xi)  y(\xi)  + q(\bar{x},\xi))_i=0\}.
\end{aligned}
\end{equation}

Denote by $\mathcal{L}(\mathbb{R}^m)$ the collection of all measurable mappings from $\Xi$ to $\mathbb{R}^m$.

\begin{definition}[stationary points of problem \eqref{MPEC-1}]\label{Def7}
(i) We say that $\bar{y}$ is a W-stationary point of problem \eqref{MPEC-1} if there exist multipliers $(\lambda,\mu)\in\mathcal{L}(\mathbb{R}^m)\times\mathcal{L}(\mathbb{R}^m)$ such that for a.e. $\xi\in\Xi$,
\begin{equation}
\label{gs19}
\begin{aligned}
&\nabla h(\mathbb{E}_{\bar{P}}[f(\bar{x},\bar{y}(\xi),\xi)]) \nabla_y f(\bar{x},\bar{y}(\xi),\xi) - \lambda(\xi) - M(\xi)^\top \mu(\xi) = 0,\\
& \lambda_i(\xi) =0 ~\text{for}~ i\in \mathcal{I}_{+0}(\bar{y};\xi),~ \mu_i(\xi) =0 ~\text{for}~ i\in \mathcal{I}_{0+}(\bar{y};\xi).
\end{aligned}
\end{equation}

(ii) We say that $\bar{y}$ is a C-stationary point of problem \eqref{MPEC-1} if there exist $(\lambda,\mu)\in\mathcal{L}(\mathbb{R}^m)\times\mathcal{L}(\mathbb{R}^m)$ satisfying \eqref{gs19} and  for a.e. $\xi\in\Xi$,
$$\lambda_i(\xi)\mu_i(\xi)\geq 0~\text{for}~ i\in \mathcal{I}_{00}(\bar{y};\xi).$$

(iii) We say that $\bar{y}$ is an M-stationary point of problem \eqref{MPEC-1} if there exist $(\lambda,\mu)\in\mathcal{L}(\mathbb{R}^m)\times\mathcal{L}(\mathbb{R}^m)$ satisfying \eqref{gs19} and  for a.e. $\xi\in\Xi$,
$$\lambda_i(\xi)>0,~\mu_i(\xi)>0~\text{or}~ \lambda_i(\xi)\mu_i(\xi) = 0 ~\text{for}~ i\in \mathcal{I}_{00}(\bar{y};\xi).$$

(iv) We say that $\bar{y}$ is an S-stationary point of problem \eqref{MPEC-1} if there exist $(\lambda,\mu)\in\mathcal{L}(\mathbb{R}^m)\times\mathcal{L}(\mathbb{R}^m)$ satisfying \eqref{gs19} and  for a.e. $\xi\in\Xi$,
$$\lambda_i(\xi)\geq 0,~\mu_i(\xi)\geq 0~\text{for}~ i\in \mathcal{I}_{00}(\bar{y};\xi).$$
\end{definition}

Note that when $\Xi$ has a finite number of elements, Definition \ref{Def7} reduces to  Definition \ref{Def5}.

\begin{definition}[block coordinatewise $\bullet$-stationary points of \eqref{DREC}]
$(\bar{x}, \bar{y},\bar{P})$ is called a block coordinatewise $\bullet$-stationary point of problem \eqref{DREC} if it satisfies
\begin{equation*}
\begin{cases}
0 \in \nabla_x \theta(\bar{x}) + \nabla h(\mathbb{E}_{\bar{P}}[f(\bar{x},\bar{y}(\xi),\xi)])\mathbb{E}_{\bar{P}}[\nabla_x f(\bar{x},\bar{y}(\xi),\xi)] + \mathcal{N}_X(\bar{x}),\\
\bar{y} ~\text{is a $\bullet$-stationary point of problem \eqref{MPEC-1}},\\
0 \in  -\nabla h(\mathbb{E}_{\bar{P}}[f(\bar{x},\bar{y}(\xi),\xi)]) f(\bar{x},\bar{y}(\cdot),\cdot) + \mathcal{N}_\mathfrak{P}(\bar{p}),
\end{cases}
\end{equation*}
where ``$\bullet$" can be W, C, M and S, $\bar{p}$ is the density function of $\bar{P}$ and $\mathcal{N}_\mathfrak{P}(\bar{p}):=\{v\in \mathcal{L}(\mathbb{R}): \int_\Xi v(\xi)(p(\xi) - \bar{p})\d\xi \leq 0, \forall p \in \mathfrak{P}\}$.
\end{definition}

\subsection{Convergence analysis}\label{Sec4-2}

In this subsection, we study the convergence of block coordinatewise stationary points of \eqref{DREC-1} defined by \eqref{gs17}. We first consider the convergence of the block coordinatewise stationary points as $\epsilon\downarrow 0$ for a fixed $k$ in the following theorem.

\begin{theorem}\label{Th5}
Let $(x_\epsilon,\mathbf{y}_\epsilon,p_\epsilon)$ be a block coordinatewise stationary point of problem \eqref{DREC-1} defined by \eqref{gs17} and $(x^*,\mathbf{y}^*,p^*)$ be an accumulation point of $(x_\epsilon,\mathbf{y}_\epsilon,p_\epsilon)$ as $\epsilon\downarrow 0$. Suppose further that MPEC-LICQ holds at $\mathbf{y}^*$ for problem \eqref{gs8} with $(\bar{x},\bar{p})=(x^*,p^*)$ and $\epsilon=0$.  Then $(x^*,\mathbf{y}^*,p^*)$ is a block coordinatewise C-stationary point of problem \eqref{DDREC}.
\end{theorem}

\begin{proof}
Since $(x^*,\mathbf{y}^*,p^*)$ is an accumulation point of $(x_\epsilon,\mathbf{y}_\epsilon,p_\epsilon)$ as $\epsilon\downarrow 0$, there exists a sequence $\{\epsilon_j\}_{j= 1}^\infty$ with $\epsilon_j\downarrow 0$ as $j\rightarrow \infty$, such that $(x_{\epsilon_j},\mathbf{y}_{\epsilon_j},p_{\epsilon_j})\rightarrow (x^*,\mathbf{y}^*,p^*)$ as $j\rightarrow \infty$.
Based on \eqref{gs17}, we have
\begin{equation*}
\begin{cases}
0\in \nabla_x G(x_{\epsilon_j},\mathbf{y}_{\epsilon_j},p_{\epsilon_j}) + \mathcal{N}_X(x_{\epsilon_j}),\\
\mathbf{y}_{\epsilon_j}=\hat{\mathbf{y}}_{\epsilon_j}(x_{\epsilon_j}),\\
0\in -\nabla_{p}G(x_{\epsilon_j},\mathbf{y}_{\epsilon_j},p_{\epsilon_j}) + \mathcal{N}_{\mathcal{P}_k}(p_{\epsilon_j}).
\end{cases}
\end{equation*}
We know from Definitions \ref{Def5} and \ref{Def3} and Lemma \ref{Lem3} that
\begin{equation*}
\label{gs15}
\begin{cases}
0\in \nabla_x G(x_{\epsilon_j},\mathbf{y}_{\epsilon_j},p_{\epsilon_j}) + \mathcal{N}_X(x_{\epsilon_j}),\\
\begin{cases}
\nabla_\mathbf{y} G(x_{\epsilon_j},\mathbf{y}_{\epsilon_j},p_{\epsilon_j}) - \lambda^j - (\mathbf{M}^\top +\epsilon_j I) \mu^j = 0,\\
\lambda_i^j =0 ~\text{for}~ i\in \mathcal{I}_{+0}^j(\mathbf{y}_{\epsilon_j}),~ \mu_i^j =0 ~\text{for}~ i\in \mathcal{I}_{0+}^j(\mathbf{y}_{\epsilon_j}),\\
\lambda_i^j\geq0,~\mu_i^j\geq0~\text{for}~ i\in \mathcal{I}_{00}^j(\mathbf{y}_{\epsilon_j}),
\end{cases}\\
0\in -\nabla_{p}G(x_{\epsilon_j},\mathbf{y}_{\epsilon_j},p_{\epsilon_j}) + \mathcal{N}_{\mathcal{P}_k}(p_{\epsilon_j}),
\end{cases}
\end{equation*}
where $\{\lambda^j\}_{j=1}^\infty$ and $\{\mu^j\}_{j=1}^\infty$ are two sequences of multipliers, and
\begin{align*}
\mathcal{I}_{+0}^j(\mathbf{y}_{\epsilon_j})&=\{i:(\mathbf{y}_{\epsilon_j})_i>0, ((\mathbf{M} + \epsilon_j I) \mathbf{y}_{\epsilon_j}  + \mathbf{q}(x_{\epsilon_j}))_i=0\},\\
\mathcal{I}_{0+}^j(\mathbf{y}_{\epsilon_j})&=\{i:(\mathbf{y}_{\epsilon_j})_i=0, ((\mathbf{M} + \epsilon_j I) \mathbf{y}_{\epsilon_j}  + \mathbf{q}(x_{\epsilon_j}))_i>0\},\\
\mathcal{I}_{00}^j(\mathbf{y}_{\epsilon_j})&=\{i:(\mathbf{y}_{\epsilon_j})_i=0, ((\mathbf{M} + \epsilon_j I) \mathbf{y}_{\epsilon_j}  + \mathbf{q}(x_{\epsilon_j}))_i=0\}.
\end{align*}
Thus, for sufficiently large $j$, we have
\begin{equation}
\label{gs14}
\begin{cases}
0\in \nabla_x G(x_{\epsilon_j},\mathbf{y}_{\epsilon_j},p_{\epsilon_j}) + \mathcal{N}_X(x_{\epsilon_j}),\\
\begin{cases}
\nabla_\mathbf{y} G(x_{\epsilon_j},\mathbf{y}_{\epsilon_j},p_{\epsilon_j}) - \lambda^j - (\mathbf{M}^\top+\epsilon_j I) \mu^j = 0,\\
\lambda_i^j =0 ~\text{for}~ i\in \mathcal{I}_{+0}(\mathbf{y}^*),~ \mu_i^j =0 ~\text{for}~ i\in \mathcal{I}_{0+}(\mathbf{y}^*),
\end{cases}\\
0\in -\nabla_{p}G(x_{\epsilon_j},\mathbf{y}_{\epsilon_j},p_{\epsilon_j}) + \mathcal{N}_{\mathcal{P}_k}(p_{\epsilon_j}),
\end{cases}
\end{equation}
where
\begin{align*}
\mathcal{I}_{+0}(\mathbf{y}^*)&=\{i:\mathbf{y}^*_i>0, (\mathbf{M} \mathbf{y}^*  + \mathbf{q}(x^*))_i=0\},\\
\mathcal{I}_{0+}(\mathbf{y}^*)&=\{i:\mathbf{y}^*_i=0, (\mathbf{M} \mathbf{y}^*  + \mathbf{q}(x^*))_i>0\}.
\end{align*}

Next, we verify the boundedness of sequences $\{\lambda^j\}_{j=1}^\infty$ and $\{\mu^j\}_{j=1}^\infty$. Notice that if $\{\mu^j\}_{j=1}^\infty$ is bounded, from the boundedness of $\{(x_{\epsilon_j},\mathbf{y}_{\epsilon_j},p_{\epsilon_j})\}_{j=1}^\infty$ and continuous differentiability of $G$, $\{\lambda^j\}_{j=1}^\infty$ is bounded. Now we assume that $\{\mu^j\}_{j=1}^\infty$ is unbounded. Consider, by dividing $\norm{\mu^j}$, that
\begin{equation*}
\frac{ \nabla_\mathbf{y} G(x_{\epsilon_j},\mathbf{y}_{\epsilon_j},p_{\epsilon_j}) }{\norm{\mu^j}} - \frac{\lambda^j}{\norm{\mu^j}} - (\mathbf{M}^\top+\epsilon_j I) \frac{\mu^j}{\norm{\mu^j}} = 0,
\end{equation*}
which can deduce, according to $\norm{\mu^j}\rightarrow\infty$ as $j\rightarrow\infty$, that
\begin{equation}
\label{gs13}
\frac{\lambda^j}{\norm{\mu^j}} + \mathbf{M}^\top \frac{\mu^j}{\norm{\mu^j}} \rightarrow 0
\end{equation}
as $j\rightarrow\infty$. Since $\lambda_i^j =0$ for $i\in \mathcal{I}_{+0}(\mathbf{y}^*)$ and $\mu_i^j =0$ for $i\in \mathcal{I}_{0+}(\mathbf{y}^*)$, we rewrite \eqref{gs13} as
\begin{equation*}
\sum_{i\in \mathcal{I}_{0+}(\mathbf{y}^*)\cup \mathcal{I}_{00}(\mathbf{y}^*)}\frac{\lambda_i^j}{\norm{\mu^j}} \mathbf{e}_i + \sum_{i\in \mathcal{I}_{+0}(\mathbf{y}^*)\cup \mathcal{I}_{00}(\mathbf{y}^*)} \frac{\mu_i^j}{\norm{\mu^j}} \mathbf{M}^\top_i  \rightarrow 0,
\end{equation*}
where
$\mathcal{I}_{00}(\mathbf{y}^*)=\{i:\mathbf{y}^*_i=0, (\mathbf{M} \mathbf{y}^*  + \mathbf{q}(x^*))_i=0\}.$

Then, by MPEC-LICQ at $\mathbf{y}^*$ for problem \eqref{gs8} with $(\bar{x},\bar{p})=(x^*,p^*)$ and $\epsilon=0$, we obtain
$$\frac{\lambda_i^j}{\norm{\mu^j}}\rightarrow 0, i\in \mathcal{I}_{0+}(\mathbf{y}^*)\cup \mathcal{I}_{00}(\mathbf{y}^*) ~\text{and}~ \frac{\mu_i^j}{\norm{\mu^j}}\rightarrow 0, i\in \mathcal{I}_{+0}(\mathbf{y}^*)\cup \mathcal{I}_{00}(\mathbf{y}^*)$$
as $k\rightarrow\infty$, which contradicts $\frac{\mu_i^j}{\norm{\mu^j}}\nrightarrow 0$ for some $i\in \mathcal{I}_{+0}(\mathbf{y}^*)\cup \mathcal{I}_{00}(\mathbf{y}^*)$. Therefore, both $\{\lambda^j\}_{j=1}^\infty$ and $\{\mu^j\}_{j=1}^\infty$
are bounded. Without loss of generality, we assume that $\lambda^j\rightarrow \lambda^*$ and $\mu^j\rightarrow \mu^*$ as $j\rightarrow\infty$. Therefore, by letting $j\rightarrow \infty$, we have from \eqref{gs14} that
\begin{equation}
\label{gs16}
\begin{cases}
0\in \nabla_x G(x^*,\mathbf{y}^*,p^*) + \mathcal{N}_X(x^*),\\
\begin{cases}
\nabla_\mathbf{y} G(x^*,\mathbf{y}^*,p^*) - \lambda^* - \mathbf{M}^\top \mu^* = 0,\\
\lambda_i^* =0 ~\text{for}~ i\in \mathcal{I}_{+0}(\mathbf{y}^*),~ \mu_i^* =0 ~\text{for}~ i\in \mathcal{I}_{0+}(\mathbf{y}^*),
\end{cases}\\
0\in -\nabla_{p}G(x^*,\mathbf{y}^*,p^*) + \mathcal{N}_{\mathcal{P}_k}(p^*).
\end{cases}
\end{equation}
Moreover, for $\lambda_i^j\mu_i^j\geq 0$ for $i=1,\ldots, mk$, we have
 $\lambda_i^*\mu_i^*\geq 0$ for $i\in \mathcal{I}_{00}(\mathbf{y}^*)$.
 This together with \eqref{gs16} means that $(x^*,\mathbf{y}^*,p^*)$ is a block coordinatewise C-stationary point of problem \eqref{DDREC}.
\qed
\end{proof}

Let $\Xi^k:=\{\xi^1,\cdots,\xi^k\}$ and its corresponding Voronoi tessellation be $\Xi_1,\cdots,\Xi_k$.  For any feasible point of \eqref{DDREC}, denoted by $(x^k, \mathbf{y}^k, p^k)$, we make the following notations. Define the following density function:
\begin{equation*}
\mathbf{p}^k(\xi)=\sum_{i=1}^k \frac{p_i^k}{\int_{\Xi_i}1\d\xi} \mathbf{1}_{\Xi_i}(\xi) ~\text{for}~\xi\in\Xi,
\end{equation*}
where $p_i^k$ is the $i$th component of $p^k=(p^k_1,\cdots, p^k_k)^\top$ for $i=1,\cdots,k$. We denote
$$\mathfrak{P}_k=\left\{\mathbf{p}^k(\cdot)=\sum_{i=1}^k \frac{p^k_i}{\int_{\Xi_i}1\d\xi} \mathbf{1}_{\Xi_i}(\cdot): p^k\in \mathcal{P}_k \right\}.$$
Denote by
$$y^k(\cdot)=\sum_{i=1}^k y^k(\xi^i) \mathbf{1}_{\Xi_i}(\cdot),$$
where $y^k(\xi^i)$ is the $i$th block of $\mathbf{y}^k=(y^k(\xi^1)^\top,\cdots,y^k(\xi^k)^\top)^\top$ for $i=1,\cdots,k$. Denote by
\begin{equation*}
\overline{M}_k(\cdot)= \sum_{i=1}^k M(\xi^i)\mathbf{1}_{\Xi_i}(\cdot)~\text{and}~ \overline{q}_k(x,\cdot)= \sum_{i=1}^k q(x,\xi^i)\mathbf{1}_{\Xi_i}(\cdot).
\end{equation*}

If, further, $(x^k,\mathbf{y}^k,p^k)$ is a block coordinatewise C-stationary point of problem \eqref{DDREC}, according to the definition of block coordinatewise C-stationary point, we have
\begin{equation}
\label{gs22}
\begin{cases}
0\in \nabla_x G(x^k,\mathbf{y}^k,p^k) + \mathcal{N}_X(x^k),\\
\begin{cases}
\nabla_\mathbf{y} G(x^k,\mathbf{y}^k,p^k) - \lambda^k - \mathbf{M}^\top \mu^k = 0,\\
\lambda_i^k =0 ~\text{for}~ i\in \mathcal{I}_{+0}(\mathbf{y}^k),~ \mu_i^k =0 ~\text{for}~ i\in \mathcal{I}_{0+}(\mathbf{y}^k),\\
\lambda_i^k\mu_i^k\geq 0 ~\text{for}~ i\in \mathcal{I}_{00}(\mathbf{y}^k),
\end{cases}\\
0\in -\nabla_{p}G(x^k,\mathbf{y}^k,p^k) + \mathcal{N}_{\mathcal{P}_k}(p^k),
\end{cases}
\end{equation}
where
\begin{align*}
\mathcal{I}_{+0}(\mathbf{y}^k)&=\{i:\mathbf{y}^k_i>0, (\mathbf{M} \mathbf{y}^k  + \mathbf{q}(x^k))_i=0\},\\
\mathcal{I}_{0+}(\mathbf{y}^k)&=\{i:\mathbf{y}^k_i=0, (\mathbf{M} \mathbf{y}^k  + \mathbf{q}(x^k))_i>0\},\\
\mathcal{I}_{00}(\mathbf{y}^k)&=\{i:\mathbf{y}^k_i=0, (\mathbf{M} \mathbf{y}^k  + \mathbf{q}(x^k))_i=0\}.
\end{align*}

Denote by
\begin{equation*}
\lambda^k(\cdot)= \sum_{i=1}^k \lambda^k(\xi^i)\mathbf{1}_{\Xi_i}(\cdot)~\text{and}~ \mu^k(\cdot)= \sum_{i=1}^k \mu^k(\xi^i)\mathbf{1}_{\Xi_i}(\cdot),
\end{equation*}
where $\lambda^k(\xi^i)$ is the $i$th block of $\lambda^k=(\lambda^k(\xi^1)^\top,\cdots,\lambda^k(\xi^k)^\top)^\top$ and $\mu^k(\xi^i)$ is the $i$th block of $\mu^k=(\mu^k(\xi^1)^\top,\cdots,\mu^k(\xi^k)^\top)^\top$ for $i=1,\cdots,k$.

Then, we have the reformulation of \eqref{gs22} as follows: for every $\xi\in\Xi$,
\begin{equation}
\label{gs23}
\begin{cases}
0 \in \nabla_x \theta(x^k) + \nabla h(\mathbb{E}_{P^k}[f(x^k,y^k(\xi),\xi)])\mathbb{E}_{P^k}[\nabla_x f(x^k,y^k(\xi),\xi)] + \mathcal{N}_X(x^k),\\
\begin{cases}
\nabla h(\mathbb{E}_{P^k}[f(x^k,y^k(\xi),\xi)]) \nabla_y f(x^k,y^k(\xi),\xi) - \lambda^k(\xi) - \overline{M}_k(\xi)^\top \mu^k(\xi) = 0,\\
\lambda_i^k(\xi) =0 ~\text{for}~ i\in \mathcal{I}_{+0}(y^k;\xi),~ \mu_i^k(\xi) =0 ~\text{for}~ i\in \mathcal{I}_{0+}(y^k;\xi),\\
\lambda_i^k(\xi)\mu_i^k(\xi)\geq 0 ~\text{for}~ i\in \mathcal{I}_{00}(y^k;\xi),
\end{cases}\\
0 \in  -\nabla h(\mathbb{E}_{P^k}[f(x^k,y^k(\xi),\xi)]) f(x^k,y^k(\cdot),\cdot) + \mathcal{N}_{\mathfrak{P}_k}(\mathbf{p}^k),
\end{cases}
\end{equation}
where $P^k$ is the probability distribution of density function $\mathbf{p}^k$ and
\begin{align*}
\mathcal{I}_{+0}(y^k;\xi)&=\{i:y_i^k(\xi)>0, (\overline{M}_k(\xi)  y^k(\xi)  + \overline{q}_k(x^k,\xi))_i=0\},\\
\mathcal{I}_{0+}(y^k;\xi)&=\{i:y_i^k(\xi)=0, (\overline{M}_k(\xi)  y^k(\xi)  + \overline{q}_k(x^k,\xi))_i>0\},\\
\mathcal{I}_{00}(y^k;\xi)&=\{i:y_i^k(\xi)=0, (\overline{M}_k(\xi)  y^k(\xi)  + \overline{q}_k(x^k,\xi))_i=0\}.
\end{align*}

For $y_1,y_2\in \mathcal{L}(\mathbb{R}^m)$, define the inner product and its induced norm by $$\inp{y_1,y_2}=\int_{\Xi} y_1(\xi)^\top y_2(\xi)\d\xi$$
and
$$\norm{y_1 - y_2}_{\mathcal{L}_2}=\left( \int_{\Xi} \norm{y_1(\xi)- y_2(\xi)}^2 \d\xi \right)^{\frac{1}{2}}.$$
Based on $\mathcal{L}_2$-norm, we can define the convergence relationship, denoted by $\overset{\mathcal{L}_2}{\rightarrow}$, and the deviation distance, denoted by $\d_{\mathcal{L}_2}(\cdot,\cdot)$.

The following theorem claims that: under certain conditions, a sequence of C-stationary points of problem \eqref{DDREC} converges to block coordinatewise C-stationary points of problem \eqref{DREC} as $k\rightarrow \infty$.

\begin{theorem}
Let $\{(x^k,\mathbf{y}^k,p^k)\}$ be a sequence of block coordinatewise C-stationary points of problem \eqref{DDREC}. Suppose that: (i) $x^k\rightarrow \bar{x}$, $y^k \overset{\mathcal{L}_2}{\rightarrow} \bar{y}$, $\mathbf{p}^k \overset{\mathcal{L}_2}{\rightarrow} \bar{\mathbf{p}}$ and $\mu^k \overset{\mathcal{L}_2}{\rightarrow}\bar{\mu}$ as $k\rightarrow \infty$; (ii) there exists $\kappa:\Xi\rightarrow \mathbb{R}_+$ satisfying $\int_{\Xi} \kappa(\xi)^2 \d\xi < \infty$, such that $\norm{f(x^k,y^k(\xi),\xi)}\leq \kappa(\xi)$, $\norm{f(\bar{x},\bar{y}(\xi),\xi)}\leq \kappa(\xi)$, $\norm{\nabla_x f(x^k,y^k(\xi),\xi)} \leq \kappa(\xi)$, $\norm{\nabla_x f(\bar{x},\bar{y}(\xi),\xi)} \leq \kappa(\xi)$, $\abs{\mathbf{p}^k(\xi)}\leq \kappa(\xi)$ and $\abs{\bar{\mathbf{p}}(\xi)}\leq \kappa(\xi)$ for a.e. $\xi\in\Xi$; (iii) $\beta_k \rightarrow 0$ as $k\rightarrow \infty$ where $\beta_k$ is defined in \eqref{beta}; (iv) $\d_{\mathcal{L}_2}(\mathfrak{P},\mathfrak{P}_k)\rightarrow 0$ as $k\rightarrow \infty$. Then $(\bar{x},\bar{y},\bar{\mathbf{p}})$ is a block coordinatewise C-stationary point of problem \eqref{DREC}.
\end{theorem}

\begin{proof}
Note that
\begin{align*}
&~~~~\norm{\mathbb{E}_{P^k}[f(x^k,y^k(\xi),\xi)]- \mathbb{E}_{\bar{P}}[f(\bar{x},\bar{y}(\xi),\xi)]}\\
&= \norm{\int_\Xi f(x^k,y^k(\xi),\xi) \mathbf{p}^k(\xi) \d\xi -  \int_\Xi f(\bar{x},\bar{y}(\xi),\xi) \bar{\mathbf{p}}(\xi) \d\xi }\\
&\leq  \norm{\int_\Xi ( f(x^k,y^k(\xi),\xi) - f(\bar{x},\bar{y}(\xi),\xi) ) \mathbf{p}^k(\xi) \d\xi }  \\
&~~~~+   \norm{ \int_\Xi f(\bar{x},\bar{y}(\xi),\xi)  (\mathbf{p}^k(\xi) - \bar{\mathbf{p}}(\xi)) \d\xi }.
\end{align*}
Since $y^k \overset{\mathcal{L}_2}{\rightarrow} \bar{y}$ and $\mathbf{p}^k \overset{\mathcal{L}_2}{\rightarrow} \bar{\mathbf{p}}$, we have $y^k(\xi)\rightarrow \bar{y}(\xi)$ and $\mathbf{p}^k(\xi) \rightarrow \bar{\mathbf{p}}(\xi)$ for a.e. $\xi\in \Xi$ as $k\rightarrow \infty$. By the continuity of $f$, we have
$$\norm{ f(x^k,y^k(\xi),\xi) - f(\bar{x},\bar{y}(\xi),\xi) } \abs{\mathbf{p}^k(\xi)}\rightarrow 0$$
for a.e. $\xi\in \Xi$ as $k\rightarrow \infty$. Moreover,
$$\norm{ f(x^k,y^k(\xi),\xi) - f(\bar{x},\bar{y}(\xi),\xi) } \abs{\mathbf{p}^k(\xi)}\leq 2\kappa^2(\xi).$$
By Lebesgue's dominated convergence theorem, we have
\begin{align*}
&\norm{\int_\Xi ( f(x^k,y^k(\xi),\xi) - f(\bar{x},\bar{y}(\xi),\xi) ) \mathbf{p}^k(\xi) \d\xi } \\
& \leq \int_\Xi \norm{ f(x^k,y^k(\xi),\xi) - f(\bar{x},\bar{y}(\xi),\xi) } \abs{\mathbf{p}^k(\xi)} \d\xi  \rightarrow  0
\end{align*}
as $k\rightarrow \infty$, and similarly
\begin{align*}
\norm{ \int_\Xi f(\bar{x},\bar{y}(\xi),\xi)  (\mathbf{p}^k(\xi) - \bar{\mathbf{p}}(\xi)) \d\xi } &\leq  \int_\Xi \norm{f(\bar{x},\bar{y}(\xi),\xi)}  \abs{\mathbf{p}^k(\xi) - \bar{\mathbf{p}}(\xi)} \d\xi \\
&\rightarrow  0
\end{align*}
as $k\rightarrow \infty$.
Therefore,
$\mathbb{E}_{P^k}[f(x^k,y^k(\xi),\xi)] \rightarrow \mathbb{E}_{\bar{P}}[f(\bar{x},\bar{y}(\xi),\xi)]$ as $k\rightarrow \infty$.
Due to the continuous differentiability of $h$, we obtain
\begin{equation}
\label{gs26}
\nabla h(\mathbb{E}_{P^k}[f(x^k,y^k(\xi),\xi)]) \rightarrow  \nabla h(\mathbb{E}_{\bar{P}}[f(\bar{x},\bar{y}(\xi),\xi)]) ~\text{as}~k\rightarrow \infty.
\end{equation}
Analogously, we have
\begin{align*}
&~~~~\norm{\mathbb{E}_{P^k}[\nabla_x f(x^k,y^k(\xi),\xi)] -  \mathbb{E}_{\bar{P}}[\nabla_x f(\bar{x},\bar{y}(\xi),\xi)]}\\
&= \norm{\int_\Xi \nabla_x f(x^k,y^k(\xi),\xi) \mathbf{p}^k(\xi) \d\xi -  \int_\Xi \nabla_x f(\bar{x},\bar{y}(\xi),\xi) \bar{\mathbf{p}}(\xi) \d\xi }\\
&\leq \norm{\int_\Xi ( \nabla_x f(x^k,y^k(\xi),\xi) - \nabla_x f(\bar{x},\bar{y}(\xi),\xi) ) \mathbf{p}^k(\xi) \d\xi }  \\
&~~~~+   \norm{ \int_\Xi \nabla_x f(\bar{x},\bar{y}(\xi),\xi)  (\mathbf{p}^k(\xi) - \bar{\mathbf{p}}(\xi)) \d\xi } \\
&  \rightarrow  0
\end{align*}
as $k\rightarrow \infty$. We obtain
\begin{equation}
\label{gs27}
\mathbb{E}_{P^k}[\nabla_x f(x^k,y^k(\xi),\xi)] \rightarrow  \mathbb{E}_{\bar{P}}[\nabla_x f(\bar{x},\bar{y}(\xi),\xi)] ~\text{as}~k\rightarrow \infty.
\end{equation}

Thus, by letting $k\rightarrow \infty$, we obtain from \eqref{gs26}, \eqref{gs27} and the first equation of \eqref{gs23} that
\begin{equation}
\label{gs21}
0 \in \nabla_x \theta(\bar{x}) + \nabla h(\mathbb{E}_{\bar{P}}[f(\bar{x},\bar{y}(\xi),\xi)])\mathbb{E}_{\bar{P}}[\nabla_x f(\bar{x},\bar{y}(\xi),\xi)] + \mathcal{N}_X(\bar{x}).
\end{equation}

Since $M(\cdot)$, $q(\cdot,\cdot)$ are continuous and $\beta_k\rightarrow 0$ as $k\rightarrow\infty$, $\overline{M}_k(\xi)\rightarrow M(\xi)$ and $\overline{q}_k(x^k,\xi)\rightarrow q(\bar{x},\xi)$ as $k\rightarrow \infty$.
By directly letting $k\rightarrow \infty$, we obtain from the second part of \eqref{gs23} that
$$\nabla h(\mathbb{E}_{\bar{P}}[f(\bar{x},\bar{y}(\xi),\xi)]) \nabla_y f(\bar{x},\bar{y}(\xi),\xi) - \bar{\lambda}(\xi) - M(\xi)^\top \bar{\mu}(\xi) = 0$$
for a.e. $\xi\in\Xi$. Here $\bar{\lambda}(\xi)$ is the limit of $\{\lambda^k(\xi)\}$ in the sense of $\mathcal{L}_2$-norm as $k\rightarrow \infty$. Its existence is due to the convergence of sequences $\{x^k\}$, $\{y^k\}$ and $\{\mu^k\}$ as $k\to \infty$. Let $\mathcal{I}_{+0}(\bar{y};\xi)$, $\mathcal{I}_{0+}(\bar{y};\xi)$ and $\mathcal{I}_{00}(\bar{y};\xi)$ be denoted in \eqref{gs3}.
Obviously, we have for a.e. $\xi\in\Xi$ that
$\mathcal{I}_{+0}(\bar{y};\xi) \subseteq \mathcal{I}_{+0}(y^k;\xi)$ and $\mathcal{I}_{0+}(\bar{y};\xi) \subseteq \mathcal{I}_{0+}(y^k;\xi)$ for sufficiently large $k$. Thus, we have
$$\bar{\lambda}_i(\xi) =0 ~\text{for}~ i\in \mathcal{I}_{+0}(\bar{y};\xi)~\text{and}~ \bar{\mu}_i(\xi) =0 ~\text{for}~ i\in \mathcal{I}_{0+}(\bar{y};\xi).$$
By a similar discussion as the proof of Theorem \ref{Th5}, we obtain $\bar{\lambda}_i(\xi)\bar{\mu}_i(\xi)\geq 0 ~\text{for}~ i\in \mathcal{I}_{00}(\bar{y};\xi)$. To sum up, we obtain
\begin{equation}
\label{gs24}
\begin{cases}
\nabla h(\mathbb{E}_{\bar{P}}[f(\bar{x},\bar{y}(\xi),\xi)]) \nabla_y f(\bar{x},\bar{y}(\xi),\xi) - \bar{\lambda}(\xi) - M(\xi)^\top \bar{\mu}(\xi) = 0,\\
\bar{\lambda}_i(\xi) =0 ~\text{for}~ i\in \mathcal{I}_{+0}(\bar{y};\xi),~ \bar{\mu}_i(\xi) =0 ~\text{for}~ i\in \mathcal{I}_{0+}(\bar{y};\xi),\\
\bar{\lambda}_i(\xi)\bar{\mu}_i(\xi)\geq 0 ~\text{for}~ i\in \mathcal{I}_{00}(\bar{y}(\xi);\xi).
\end{cases}
\end{equation}

Since $\d_{\mathcal{L}_2}(\mathfrak{P},\mathfrak{P}_k)\rightarrow 0$ as $k\rightarrow \infty$, for any $\rho\in \mathfrak{P}$, there exists $\rho^k\in \mathfrak{P}_k$ such that $\rho^k\overset{\mathcal{L}_2}{\rightarrow} \rho$ as $k\rightarrow\infty$. Thus, we have
\begin{equation}
\label{gs2}
\rho^k - \mathbf{p}^k \overset{\mathcal{L}_2}{\rightarrow} \rho - \bar{\mathbf{p}} ~\text{and}~ f(x^k,y^k(\cdot),\cdot) \overset{\mathcal{L}_2}{\rightarrow} f(\bar{x},\bar{y}(\cdot),\cdot)
\end{equation}
as $k\rightarrow\infty$. Due to
$$0 \in  -\nabla h(\mathbb{E}_{P^k}[f(x^k,y^k(\xi),\xi)]) f(x^k,y^k(\cdot),\cdot) + \mathcal{N}_{\mathfrak{P}_k}(\mathbf{p}^k),$$
we have $\inp{\nabla h(\mathbb{E}_{P^k}[f(x^k,y^k(\xi),\xi)]) f(x^k,y^k(\cdot),\cdot),   \rho^k - \mathbf{p}^k} \leq 0.$ Based on \eqref{gs26} and \eqref{gs2}, by letting $k\to\infty$, we obtain
$$\inp{\nabla h(\mathbb{E}_{\bar{P}}[f(\bar{x},\bar{y}(\xi),\xi)]) f(\bar{x},\bar{y}(\cdot),\cdot), \rho - \bar{\mathbf{p}}} \leq 0$$
for any $\rho\in \mathfrak{P}$, which is equivalent to
\begin{equation}
\label{gs25}
0 \in  -\nabla h(\mathbb{E}_{\bar{P}}[f(\bar{x},\bar{y}(\xi),\xi)]) f(\bar{x},\bar{y}(\cdot),\cdot) + \mathcal{N}_\mathfrak{P}(\bar{\mathbf{p}}).
\end{equation}

Finally, combined \eqref{gs21} and \eqref{gs24} with  \eqref{gs25}, we complete the proof.
\qed

\end{proof}

\section{Pure characteristics demand models} \label{Sec5}

In this section, we apply our results established in the previous sections to the pure characteristics demand model when the underlying probability distribution is uncertain. In particular, we consider the distributionally robust counterpart of problem \eqref{CSW-1}:
\begin{equation*}
\begin{array}{cl}
\min\limits_{x\in X, s_t}& \frac{1}{2}\inp{x,Hx} + \inp{c,x} + \varrho \max\limits_{P\in\mathcal{P}} \sum_{t=1}^T\norm{A_t\mathbb{E}_P[s_t(\xi)] - b_t}^2 \\
\mathrm{s.t.}& 0\leq z_t(\xi) \bot M z_t(\xi) + q_t(x,\xi)\geq 0, \quad  t=1,2,\cdots,T,
\end{array}
\end{equation*}
where $\mathcal{P}:=\{P\in\mathcal{P}(\Xi): (\mathbb{E}_P[\xi] - \mu_0)^\top (\mathbb{E}_P[\xi] - \mu_0) \leq \eta \}$, $\mu_0\in\mathbb{R}^\nu$  and $\eta >0$, which assumes that the mean of $\xi$ lies in a ball of size $\eta$ centered at the estimate $\mu_0$ (see e.g. \cite[(1a)]{DY2010distributionally}). By employing the Schur complement, we can reformulate $\mathcal{P}$ as the form \eqref{AS}:
$$\mathcal{P}:=\left\{P\in\mathcal{P}(\Xi): \mathbb{E}_P\left[
\begin{pmatrix}
I & \xi -\mu_0\\
(\xi -\mu_0)^\top & \eta
\end{pmatrix} \right]
\in \Gamma \right\},$$
where the cone $\Gamma$ denotes the set of positive semidefinite $(\nu+1)\times (\nu+1)$ matrices.

The regularization and discretization problem is
\begin{equation*}
\min\limits_{x\in X}  \frac{1}{2}\inp{x,Hx} + \inp{c,x} + \varrho \max\limits_{p\in\mathcal{P}_k} \sum_{t=1}^T\norm{A_t\sum_{i=1}^k s_{t,\epsilon}(x, \xi^i)p_i - b_t}^2,
\end{equation*}
where $\{\xi^1,\cdots,\xi^k\}\subseteq \Xi$, $z_{t,\epsilon}(x,\xi^i)\in \mathrm{SOL}(M + \epsilon I, q_t(x,\xi^i))$, $z_{t,\epsilon}(x, \xi^i):= (s_{t,\epsilon}(x,\xi^i)^\top, \gamma_{t,\epsilon}(x,\xi^i))^\top$ for $i=1,\cdots,k$, $t=1,\cdots,T$ and  $$\mathcal{P}_k:=\left\{p\in\mathbb{R}_+^k: \sum_{i=1}^k p_i =1,~ \norm{\sum_{i=1}^k p_i\xi^i - \mu_0}^2 \leq \eta \right\}.$$

We adopt the following settings, which follow from \cite[Example 4.1]{CSW2015regularized}.
The utility function in market $t$ is given by
$$u_t(x,\xi)=\mathbf{C}_t\chi_1(x_2,x_3,\xi_1) - \chi_2(x_4,\xi_2)\sigma_t + x_{1t},$$
where $\mathbf{C}_t=(C_{1t},\cdots,C_{\tau t})\in\mathbb{R}^{m \times \tau}$, $C_{jt}\in \mathbb{R}^{m}$, $x_{1t}\in \mathbb{R}^m$, $x_2,x_3\in \mathbb{R}^\tau$, $x_4\in \mathbb{R}$. Let $\mathbf{x}_1=(x_{11}^\top,\cdots,x_{1T}^\top)^\top\in \mathbb{R}^{mT}$, $\mathbf{x}_2=(x_2^\top,x_3^\top,x_4)^\top$ and $x=(\mathbf{x}_1^\top,\mathbf{x}_2^\top)^\top\in \mathbb{R}^n$ with $n=mT+2\tau+1$, $\xi=(\xi_1,\xi_2)^\top:\Omega\rightarrow \Xi \subseteq \mathbb{R}^{2}$,
$\chi_1(x_2,x_3,\xi_1) = x_2+x_3\xi_1$, $\chi_2(x_4,\xi_2) = \exp(x_4\xi_2)$, $\sigma_t\in\mathbb{R}^m$.

To numerically present these convergence results in Section \ref{Sec3}, we make the following specific settings.
Set $T=1$, $m=2$, $\tau=1$, $n=5$, $b_1=(0.5,0.5)^\top$, $C_1=(2,3)^\top$, $\sigma_1=(1,2)^\top$ and $X=[0,2]^5$, $\Xi=[-1,1]^2$, $\mu_0=(0,0)^\top$. Consider
\begin{equation}
\label{NM}
\begin{array}{cl}
\min\limits_{x\in X}\max\limits_{p\in\mathcal{P}_k} & G(x,\hat{\mathbf{z}}_\epsilon(x),p):=\frac{1}{2}\inp{\mathbf{x}_1,\mathbf{x}_1} +  \norm{\sum_{i=1}^k s_{1,\epsilon}(x, \xi^i)p_i - b_1}^2.
\end{array}
\end{equation}
Here $(s_{1,\epsilon}(x, \xi^i)^\top, \gamma_{1,\epsilon}(x,\xi^i))^\top =z_{1,\epsilon}(x,\xi^i)\in \mathrm{SOL}(M + \epsilon I, q_1(x,\xi^i)), ~ i=1,\cdots,k$ and $\hat{\mathbf{z}}_\epsilon(x):=(z_{1,\epsilon}(x,\xi^1)^\top, \cdots, z_{1,\epsilon}(x,\xi^k)^\top)^\top$.

It is easy to check that $x^*=(0,0,1,0,0)^\top$ is an optimal solution of problem \eqref{NM} and the corresponding optimal value is $0$. To see this, we have
$$\chi_1(x_2^*,x_3^*,\xi_1) = x_2^*+x_3^*\xi_1=1 ~\text{and}~\chi_2(x_4^*,\xi_2) = \exp(x_4^*\xi_2) = 1.$$
Therefore, $u_1(x^*,\xi)=C_1\chi_1(x_2^*,x_3^*,\xi_1) - \chi_2(x_4^*,\xi_2)\sigma_1 + \mathbf{x}_{1}^*=(1,1)^\top$. According to \eqref{UT}, we know that the solution set of $s_1(\xi)$ is $\{(\varsigma, 1-\varsigma)^\top : 0\leq \varsigma \leq 1\}$, where the least norm solution is $(0.5,0.5)^\top$ for every $\xi\in\Xi$. This implies that $\mathbb{E}_P[s_{1,\epsilon}(x^*,\xi)]=\frac{1}{2+\epsilon} (1+\epsilon,  1+\epsilon)^\top\rightarrow (0.5,0.5)^\top$ as $\epsilon\downarrow 0$. Thus, the conditions in Proposition \ref{Prop5} hold.

Under the above settings, $\Psi(\xi)=\begin{pmatrix}
I & \xi\\
\xi^\top & \eta
\end{pmatrix}$
and $\Gamma$ is the set of positive semidefinite $3\times 3$ matrices. Assumption \ref{Assu4} holds with $P_0$ being the uniform distribution over $[-1,1]^2$ and $\alpha$ being any positive scalar less than or equal to $\eta$. Moreover, Assumption \ref{Sol} holds with $\kappa(\xi) = 1$. Therefore, the convergence results in Section \ref{Sec3} hold.

We use the following alternating iterative algorithm to report some numerical results of problem \eqref{NM}.

\begin{algorithm}
\label{Alg3}
Choose an initial point $p^0 \in \mathcal{P}_k$. Let $j=0$ and do the following two steps.

\textbf{Step 1:} Generate $x^{j}$ by solving
\begin{equation}
\label{gs28}
\min_{x\in X}~ G(x,\hat{\mathbf{z}}_\epsilon(x),p^j).
\end{equation}

\textbf{Step 2:} Generate $p^{j+1}$ by solving
\begin{equation}
\label{gs29}
\max_{p\in\mathcal{P}_k}~ G(x^{j},\hat{\mathbf{z}}_\epsilon(x^j),p).
\end{equation}

Let $j=j+1$ and go to \textbf{Step 1}.
\end{algorithm}

Due to the special structure of matrix $M$, we adopt the closed-form solution in \cite{CSW2015regularized} to compute $\hat{\mathbf{z}}_\epsilon(x^j)$. The function $G(x^j,\hat{\mathbf{z}}_\epsilon(x^j),\cdot)$ is a quadratic convex function. We can use algorithm in \cite{Y1992affine} to find a maximizer of $p^j$ in (\ref{gs29})  on the bounded convex set ${\cal P}_k$. Since ${\cal P}_k$ and $X$ are bounded, the sequence $\{(x^j,p^j)\}$ generated by Algorithm \ref{Alg3} has at least one accumulation point. We employ \eqref{gs17} as the stopping criterion. Actually, we only need to verify
$
0\in \nabla_x G(x^j,\hat{\mathbf{z}}_\epsilon(x^j),p^j)+ \mathcal{N}_{X}(x^j).
$
Due to the box structure of $X$, the projection onto $X$ can be computed easily. Thus, we stop the iteration when
\begin{equation}
\label{gs30}
\norm{x^j - \proj_X( x^j-\nabla_x G(x^j,\hat{\mathbf{z}}_\epsilon(x^j),p^j) ) } \leq 10^{-4}.
\end{equation}

We chose an initial point $p^0\in {\cal P}_k$ with $p^0_i=\frac{1}{k}, i=1,\ldots,k$.  First, for fixed $\eta=0.1, 0.2, 0.5, 1$ and sample size $k=25,2500$, we compute optimal values of problem \eqref{NM} w.r.t. $\epsilon=0.5,0.2,0.1,0.05,0.01$. We present these results in Figure \ref{fig:1}. It shows the tendency that the optimal value of problem \eqref{NM} tends to zero as $\epsilon$ goes to zero. Meanwhile, for fixed each $\epsilon=0.5,0.2,0.1,0.05,0.01$, we can observe from Figure \ref{fig:1} that the optimal value of problem \eqref{NM} increases as $\eta$ increases, which shows that the distributionally robust model \eqref{NM} works as expected.

\begin{figure}[htp]
\subfigure[Convergence of optimal values as $\epsilon$ decreases for $k=25$.]{
  \label{fig:1-1}
  \includegraphics[scale=0.4]{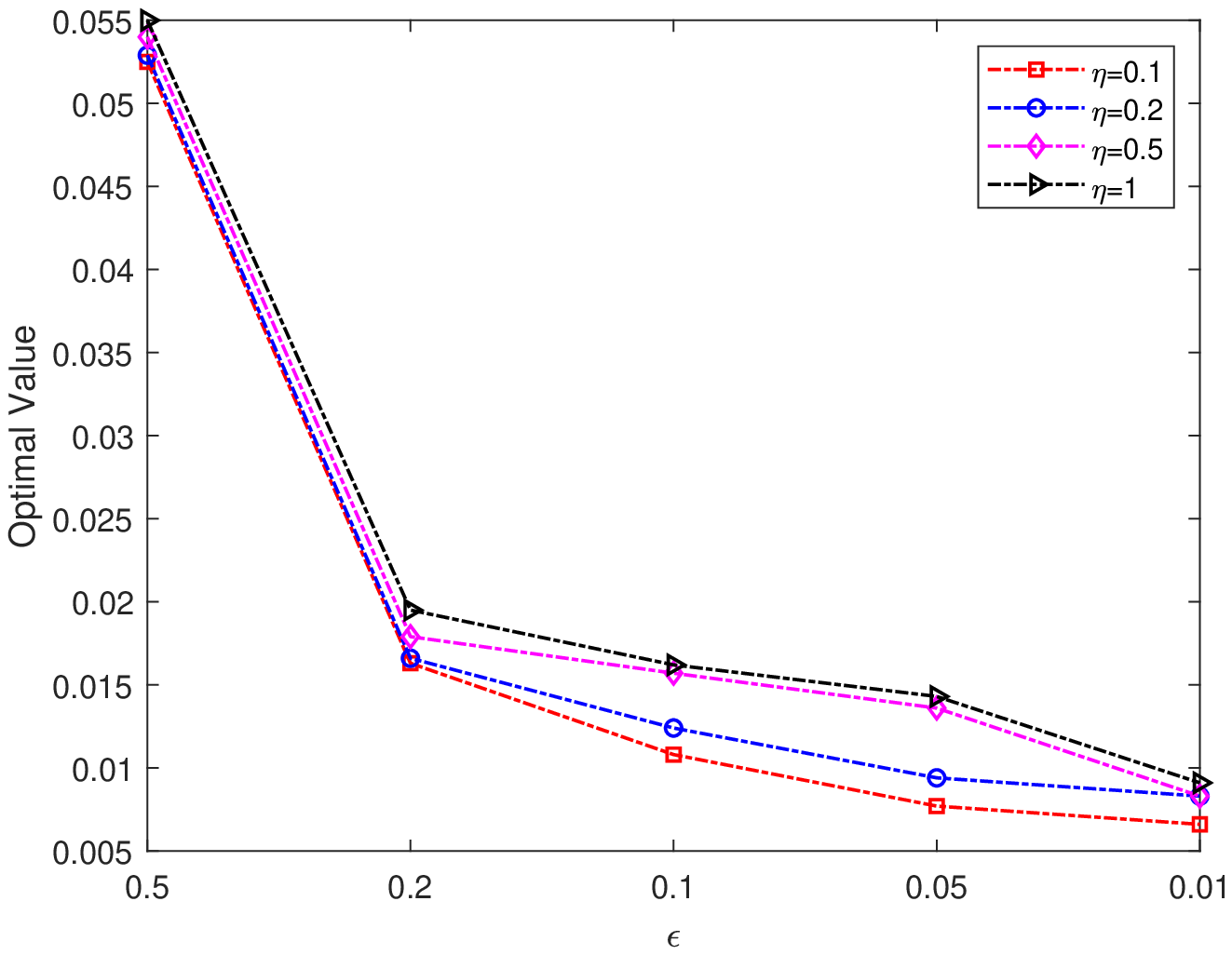}}
  \subfigure[Convergence of optimal values as $\epsilon$ decreases for $k=2500$.]{
  \label{fig:1-2}
  \includegraphics[scale=0.4]{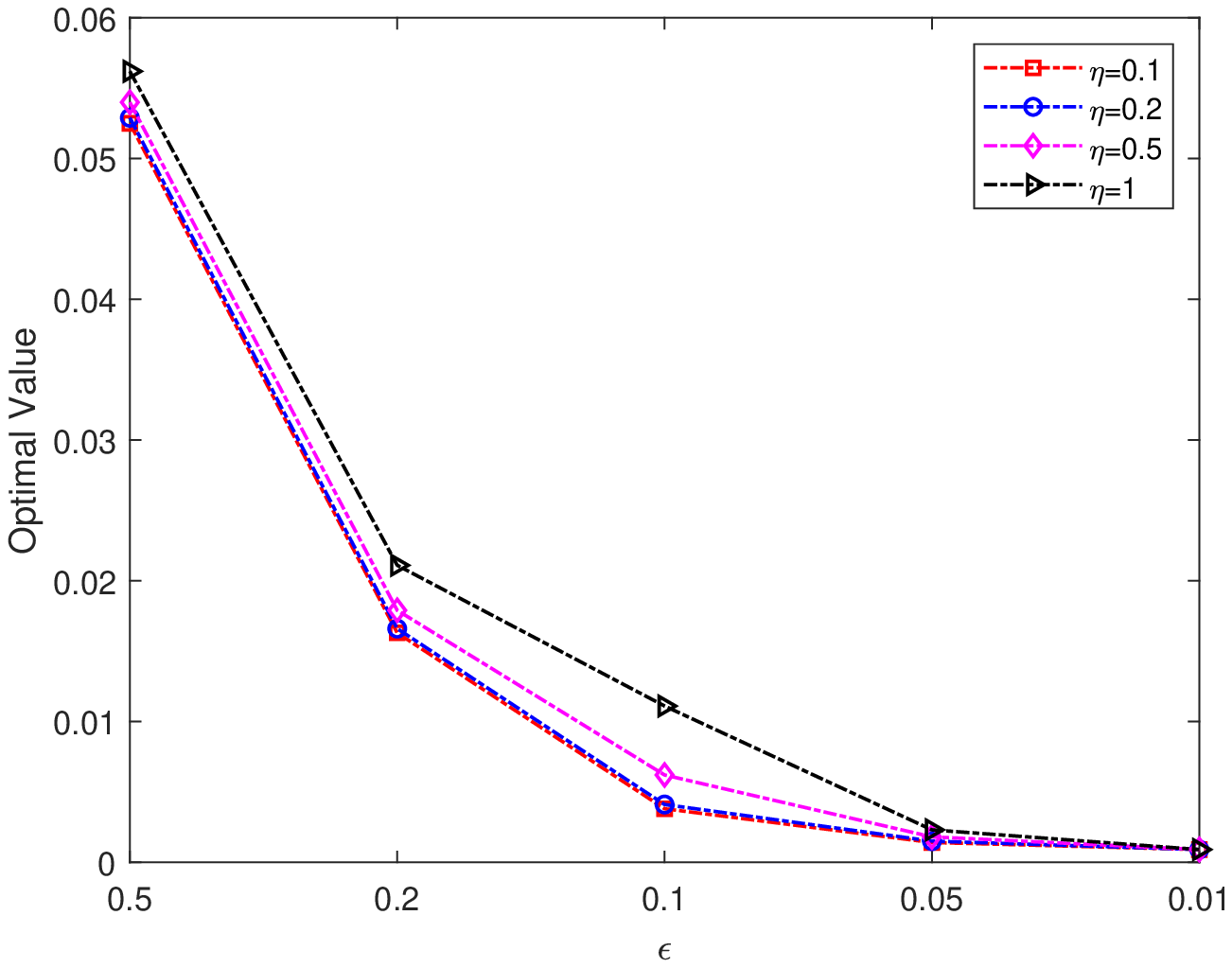}}
\caption{Convergence of optimal values as $\epsilon$ decreases.}
\label{fig:1}
\end{figure}

Furthermore, for fixed $\eta=0.1, 0.2, 0.5, 1$ and $\epsilon=0.5,0.2$, we compute optimal values of problem \eqref{NM} with different sample sizes, see Figure \ref{fig:2}. It shows that, for each fixed $\epsilon$, the optimal values converge as sample size goes to infinity. Moreover, we present in Table \ref{tab:1} the distances between $x^j$ satisfying \eqref{gs30} and the true solution $(0,0,1,0,0)^\top$ with different $\epsilon$ and $k$ for fixed $\eta=0.5$. It shows the  convergence of optimal solutions as sample size goes to infinity for fixed $\epsilon$.

\begin{figure}[htp]
\subfigure[Convergence of optimal values as sample size increases for $\epsilon=0.5$.]{
  \label{fig:2-1}
  \includegraphics[scale=0.4]{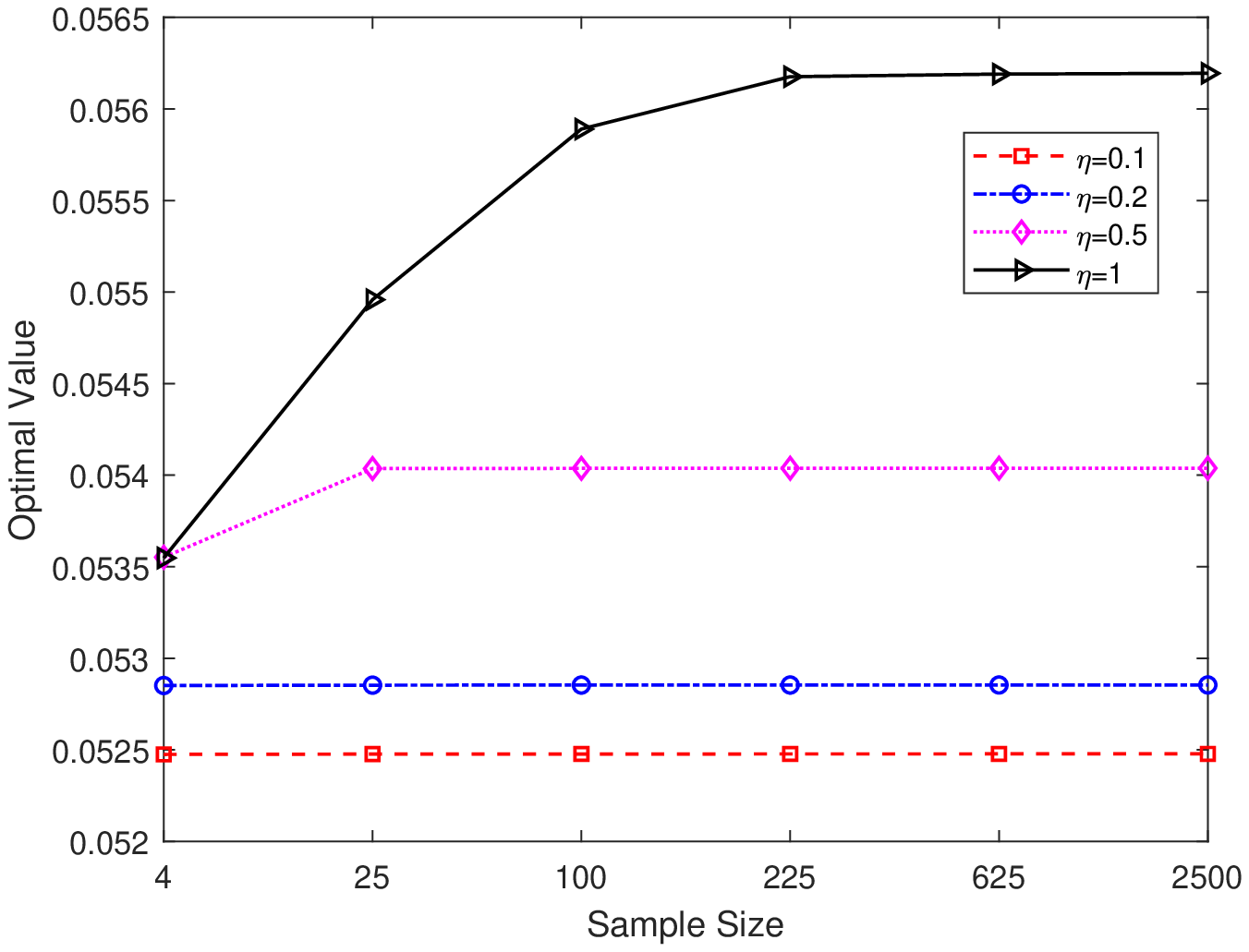}}
  \subfigure[Convergence of optimal values as sample size increases for $\epsilon=0.2$.]{
  \label{fig:2-2}
  \includegraphics[scale=0.4]{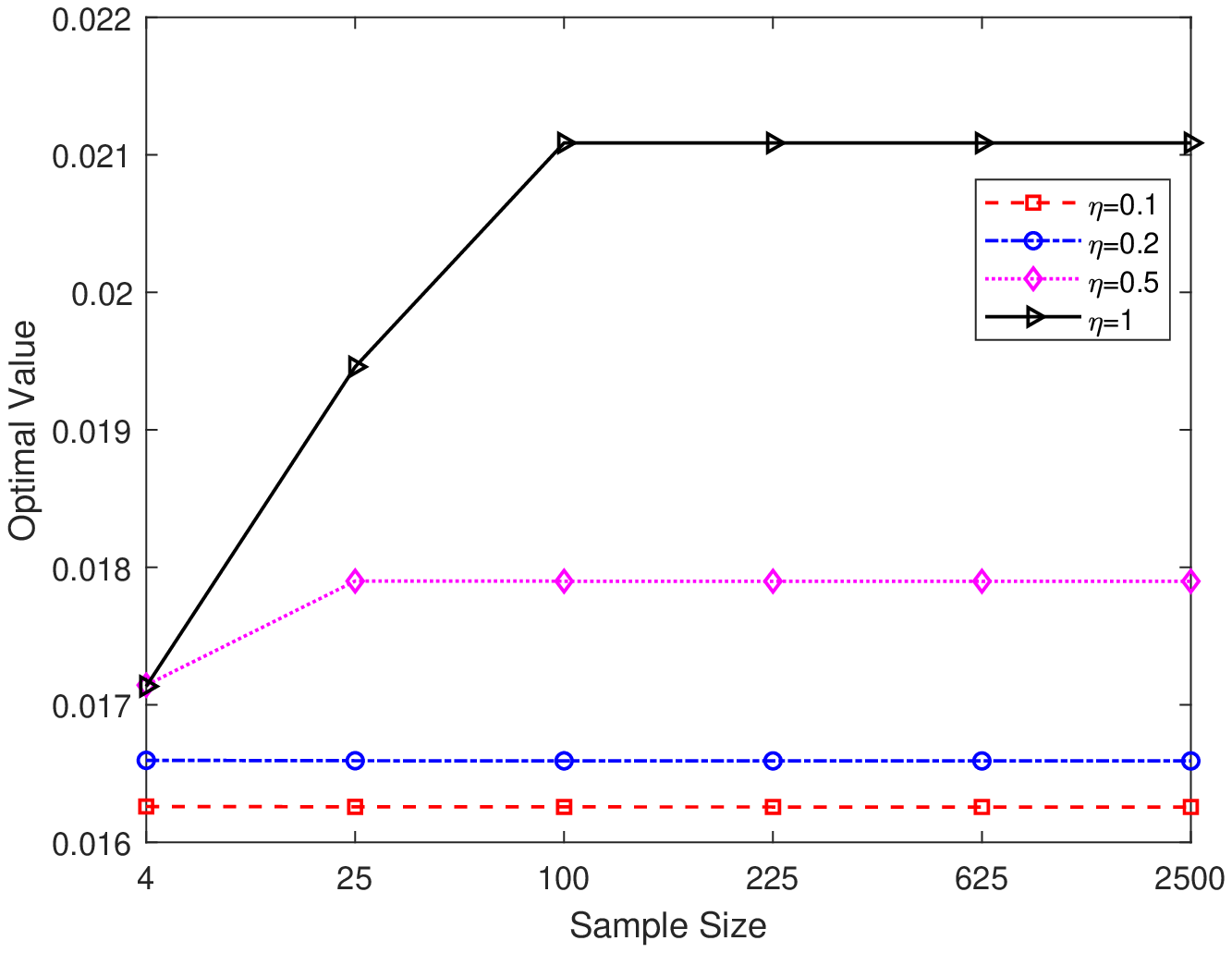}}
\caption{Optimal values w.r.t. different sample sizes.}
\label{fig:2}
\end{figure}

\begin{table}[htp]
\center
\caption{$\norm{x^j-x^*}$ with different $\epsilon$ and $k$ for fixed $\eta=0.5$
where $x^j$ satisfies \eqref{gs30} and $x^*=(0,0,1,0,0)^\top$.}
\label{tab:1}
\begin{tabular}{|c|cccccc|}
\hline
\diagbox{$\epsilon$}{$k$} & 4 & 25 & 100 & 225 & 625 & 2500 \\
\hline
0.5 & 0.5688  &  0.5210 &  0.6791 &  0.6887  &  0.7091 & 0.7202  \\
0.2 & 0.7851 & 0.4422 & 0.3195 & 0.2215 & 0.2119 & 0.2102 \\
0.1 & 0.6649  & 0.1079  & 0.0851  &  0.0853  & 0.0867  &  0.0856 \\
\hline
\end{tabular}
\end{table}

\section{Concluding remarks}\label{Sec6}

This paper considers a class of distributionally robust mathematical programs with stochastic complementarity constraints (DRMP-SCC) in the form of problem \eqref{DREC}, which arise from pure characteristics demand models under uncertainties of probability distributions of the involved random variables. Since problem \eqref{DREC} is a nonconvex-nonconcave minimax problem, minimax is not equal to maximin and thus a saddle point does not exist in general. We define global and local optimality and stationary points of problem \eqref{DREC}, and its discretization and/or regularization approximation problems \eqref{R-DREC}, \eqref{DDREC} and \eqref{DREC-1}. We provide sufficient conditions for the convergence of optimal solutions and stationary points of problem \eqref{DREC-1} as $\epsilon$ goes to zero and $k$ goes to infinity. We show that all those conditions hold for pure characteristics demand models under uncertainties. Moreover,  we use numerical results to show the effectiveness of our theoretical results.

\section*{Appendix}

\begin{proof}[The proof of Proposition \ref{Prop1}]

Denote $\bar{p}_i=P_0(\Xi_i)$ for $i=1,\cdots,k$. We verify that $\bar{p}=(\bar{p}_1,\cdots,\bar{p}_k)^\top\in\mathcal{P}_k$  for all sufficiently large $k$ in the following. Since $\Psi$ is continuous, we know from mean value theorem of integrals that
$$\mathbb{E}_{P_0}[\Psi(\xi)]=\sum_{i=1}^k \int_{\Xi_i} \Psi(\xi) \,P_0(d\xi)=\sum_{i=1}^k \Psi(\tilde{\xi}^i)P_0(\Xi_i)$$
for some $\tilde{\xi}^i\in\Xi_i$, $i=1,\cdots,k$. Then
\begin{equation}
\label{gs1}
\norm{\mathbb{E}_{P_0}[\Psi(\xi)] - \sum_{i=1}^k \bar{p}_i \Psi(\xi^i) } \leq  \sum_{i=1}^k \bar{p}_i \norm{ \Psi(\xi^i) - \Psi(\tilde{\xi}^i) }.
\end{equation}
We first consider the case that $\Xi$ is bounded. For $\alpha>0$, there exists $\delta>0$  such that if $\max_{1\leq i\leq k}\diam(\Xi_i) <\delta$, then
$$\max_{1\leq i\leq k}\norm{ \Psi(\xi^i) - \Psi(\tilde{\xi}^i) } \leq  \alpha.$$
Since $\Xi$ is bounded, we can find a sequence $\{\xi^k\}_{k=1}^\infty$ such that the corresponding Voronoi tessellation $\Xi_1,\cdots,\Xi_k,\cdots$ satisfying
$$\lim_{k\rightarrow\infty}\max_{1\leq i\leq k}\mathrm{diam}(\Xi_i)=0.$$
Hence there is $\bar{k}>0$ such that $\max_{1\leq i\leq k}\mathrm{diam}(\Xi_i)<\delta$ for any $k\ge \bar{k}$.

Then, it knows from \eqref{gs1} that
\begin{align*}
\norm{\mathbb{E}_{P_0}[\Psi(\xi)] - \sum_{i=1}^k \bar{p}_i \Psi(\xi^i) } \leq \alpha.
\end{align*}
This, together with Assumption \ref{Assu4}, indicates that
$\sum_{i=1}^k \bar{p}_i \Psi(\xi^i) \in \Gamma,$
which implies the nonemptiness of $\mathcal{P}_k$.

Now we consider the case $\Xi$ is unbounded. Let $\Xi_b:=\{\xi\in\Xi:\norm{\xi}\leq b\}$ for $b>0$.
Denote a probability distribution $\bar{P}_0$ supported on $\Xi_b$ by
$$\bar{P}_0(\Xi_a)=\frac{P_0(\Xi_a\cap \Xi_b)}{P_0(\Xi_b)}$$
for any measurable $\Xi_a\subseteq \Xi$, where $P_0$ is defined in Assumption \ref{Assu4}. Note that
$$\lim_{b\rightarrow\infty}\frac{1}{P_0(\Xi_b)}=1~\text{and}~\lim_{b\rightarrow\infty} \int_{\Xi_b} \Psi(\xi) P_0(d\xi) = \int_{\Xi} \Psi(\xi) P_0(d\xi)=\mathbb{E}_{P_0}[\Psi(\xi)].$$
We have
\begin{align*}
\lim_{b\rightarrow\infty}\int_{\Xi_b} \Psi(\xi) \bar{P}_0(d\xi) = \lim_{b\rightarrow\infty} \frac{1}{P_0(\Xi_b)}\int_{\Xi_b} \Psi(\xi) P_0(d\xi) = \mathbb{E}_{P_0}[\Psi(\xi)].
\end{align*}
Therefore, there exists $b_0>0$ such that, for any $b\geq b_0$,
$$\norm{\mathbb{E}_{\bar{P}_0}[\Psi(\xi)] - \mathbb{E}_{P_0}[\Psi(\xi)]}\leq \frac{\alpha}{2}.$$
From Assumption \ref{Assu4}, we obtain
\begin{equation}
\label{gs10}
\mathbb{E}_{\bar{P}_0}[\Psi(\xi)] + \frac{\alpha}{2}\mathbb{B}\subseteq \Gamma.
\end{equation}
Due to the boundedness of $\Xi_b$ and \eqref{gs10}, by the same proof for the case that $\Xi$
is bounded, there exists a $\bar{k}>0$  such that $\mathcal{P}_k$ is nonempty for $k\ge \bar{k}$.
\qed
\end{proof}

\end{document}